\documentclass[10pt, reqno]{amsart}
\usepackage{amssymb,amsfonts,lineno,amsthm,amscd,stmaryrd}
\usepackage[leqno]{amsmath}
\usepackage[active]{srcltx}
\usepackage[usenames,dvipsnames]{xcolor}
\usepackage[italian, textsize=tiny, color=BlueGreen, backgroundcolor=BlueGreen, linecolor=BlueGreen]{todonotes} 
\usepackage[colorlinks,pdfpagelabels,pdfstartview = FitH,bookmarksopen = true,bookmarksnumbered = true,linkcolor = black,plainpages = false,hypertexnames = false,citecolor = black]{hyperref}

\numberwithin{equation}{section} 

\newtheorem{theorem}{Theorem}[section]
\newtheorem{lemma}[theorem]{Lemma}
\newtheorem{proposition}[theorem]{Proposition}

\newtheorem{claim}{Claim}[section]
\newtheorem{definition}{Definition} [section]

\newtheorem{remark}{Remark}[section]

\def\dx{\,dx}

\def\dt{\,dt}

\newcommand{\R}{\mathbb{R}}
\newcommand{\N}{\mathbb{N}}

\newcommand{\eps}{\varepsilon}
\newcommand{\vs}{\vspace{3mm}}

\DeclareMathOperator{\loc}{loc}
\DeclareMathOperator*{\osc}{{osc}}
\newcommand{\wt}{\widetilde}

\def\mean#1{\mathchoice%
          {\mathop{\kern 0.2em\vrule width 0.6em height 0.69678ex depth -0.58065ex
                  \kern -0.8em \intop}\nolimits_{\kern -0.4em#1}}%
          {\mathop{\kern 0.1em\vrule width 0.5em height 0.69678ex depth -0.60387ex
                  \kern -0.6em \intop}\nolimits_{#1}}%
          {\mathop{\kern 0.1em\vrule width 0.5em height 0.69678ex
              depth -0.60387ex
                  \kern -0.6em \intop}\nolimits_{#1}}%
          {\mathop{\kern 0.1em\vrule width 0.5em height 0.69678ex depth -0.60387ex
                  \kern -0.6em \intop}\nolimits_{#1}}}

\def\vintslides_#1{\mathchoice%
          {\mathop{\kern 0.1em\vrule width 0.5em height 0.697ex depth -0.581ex
                  \kern -0.6em \intop}\nolimits_{\kern -0.4em#1}}%
          {\mathop{\kern 0.1em\vrule width 0.3em height 0.697ex depth -0.604ex
                  \kern -0.4em \intop}\nolimits_{#1}}%
          {\mathop{\kern 0.1em\vrule width 0.3em height 0.697ex depth -0.604ex
                  \kern -0.4em \intop}\nolimits_{#1}}%
          {\mathop{\kern 0.1em\vrule width 0.3em height 0.697ex depth -0.604ex
                  \kern -0.4em \intop}\nolimits_{#1}}}

\newcommand{\aveint}[2]{\mathchoice%
          {\mathop{\kern 0.2em\vrule width 0.6em height 0.69678ex depth -0.58065ex
                  \kern -0.8em \intop}\nolimits_{\kern -0.45em#1}^{#2}}%
          {\mathop{\kern 0.1em\vrule width 0.5em height 0.69678ex depth -0.60387ex
                  \kern -0.6em \intop}\nolimits_{#1}^{#2}}%
          {\mathop{\kern 0.1em\vrule width 0.5em height 0.69678ex depth -0.60387ex
                  \kern -0.6em \intop}\nolimits_{#1}^{#2}}%
          {\mathop{\kern 0.1em\vrule width 0.5em height 0.69678ex depth -0.60387ex
                  \kern -0.6em \intop}\nolimits_{#1}^{#2}}}

\begin{document}

\title[Boundary regularity for phase transitions]{Existence and boundary regularity\\ for degenerate phase transitions}

\author[P. Baroni]{Paolo Baroni}
\address{Department of Mathematical, Physical and Computer sciences, University of Parma\\
I-43124 Parma, Italy}
\email{paolo.baroni@unipr.it}

\author[T. Kuusi]{Tuomo Kuusi}
\address{Department of Mathematics and Systems Analysis, Aalto University, P.O. Box 1100, 00076 Aalto, Finland}
\email{tuomo.kuusi@aalto.fi}

\author[C. Lindfors]{Casimir Lindfors}
\address{Department of Mathematics and Systems Analysis, Aalto University, P.O. Box 1100, 00076 Aalto, Finland}
\email{casmir.lindfors@aalto.fi}

\author[J.M. Urbano]{Jos\'e Miguel Urbano}
\address{CMUC, Department of Mathematics, University of Coimbra, 3001-501 Coimbra, Portugal.}
\email{jmurb@mat.uc.pt}

\allowdisplaybreaks
\date{\today}


\begin{abstract}
We study the Cauchy-Dirichlet problem associated to a phase transition modeled upon the degenerate two-phase Stefan problem. We prove that weak solutions are continuous up to the parabolic boundary and quantify the continuity by deriving a modulus. As a byproduct, these {\em a priori} regularity results are used to prove the existence of a so-called physical solution.
\end{abstract}

\maketitle

\tableofcontents

\section{Introduction}

In this paper we complete the\textit{ tour de force}, initiated in \cite{BKU}, concerning the regularity of weak solutions for the degenerate ($p\geq 2$) two-phase Stefan problem \cite{Salsa12,UCon}
\begin{equation}\label{CD}
 \begin{cases}
 \partial_t \big[u+H_0(u)\big] \ni {\rm div}\, |Du|^{p-2}Du \quad&\text{in $\Omega_T$}\\[5pt]
 u= g&\text{on $\partial_p\Omega_T$\,,}
\end{cases}
\end{equation}
by proving the continuity up to the boundary. Using this regularity, we also obtain an existence result. Here, $\Omega_T:=\Omega\times(0,T]$ denotes the space-time cylinder, with $\Omega\subset\R^n$, $n\geq2$, a bounded domain, $\partial_p\Omega_T$\ is its parabolic boundary (see Paragraph \ref{Notation} for the relevant definitions), $H_0$ is the Heaviside graph centered at the origin, and $g$ is a continuous boundary datum.

The outcome of our effort is two-fold: on the one hand, we prove sharp {\em a priori} estimates for solutions of \eqref{CD}, and obtain the boundary continuity, quantified through a modulus, assuming a mild geometric condition on $\Omega$. On the other hand, we use this ``almost uniform'' modulus of continuity at the boundary, together with the interior modulus of continuity we deduced in \cite{BKU}, to build a solution to \eqref{CD}, which is continuous up to the boundary and enjoys the same modulus of continuity.

\vs

Problem \eqref{CD} when $p=2$ is the celebrated two-phase Stefan problem. The boundary continuity in this case was proven by Ziemer \cite{Ziemer},  for more general structures albeit with linear growth with respect to the gradient, but without an explicit, uniform modulus of continuity. This would be provided by DiBenedetto who, in \cite{DiB3}, proved the uniform continuity up to the boundary for solutions to \eqref{CD} (more precisely, for the forthcoming \eqref{the general equation} for $p=2$, which also takes into account lower order terms) with the modulus of continuity being of iterated logarithmic type in the particular case of H\"older continuous boundary datum.

Our goal is to extend the result to the degenerate case $p>2$ and to provide, already in the non-degenerate case $p=2$, a more transparent proof of the reduction of the oscillation at the lateral boundary.

More generally, we shall consider the extension of~\eqref{CD}$_1$
\begin{equation} \label{the general equation}
\partial_t \big[\beta(u)+ H_a(\beta(u))\big] \ni \mathrm{div}\, \mathcal{A}(x,t,u,Du) \quad\text{ in $\;\Omega_T$},
\end{equation}
where $\beta$ is a sufficiently smooth function, the Heaviside graph centered at $a\in\R$ is defined by 
\begin{equation}\label{Heavi}
H_a(s)=
\begin{cases}
0 & \text{if $s<a$}\,\\[1pt]
[0,1] & \text{if $s=a$}\,\\[1pt]
1 & \text{if $s> a$},
 \end{cases}
\end{equation}
and the vector field $\mathcal{A}$ satisfies the usual $p$-growth conditions (see Paragraph \ref{assump} for the exact assumptions). Our first result reads as follows.

\begin{theorem}\label{main}
Under the assumptions described in Paragraph \ref{assump}, given a boundary datum $g\in C(\partial_p\Omega_T)$, there exists $u \in C\left( \overline{\Omega_T} \right)$ solving the Cauchy-Dirichlet problem for \eqref{the general equation}, in the sense that $u$ is a local weak solution of the equation and $u=g$ on $\partial_p\Omega_T$. We call $u$ a physical solution.
\end{theorem}

We remark that the solution we build has the interior modulus of continuity described in \cite{BKU}, where we assumed the existence of a solution built in the way described in this paper. Our other main result concerns a precise modulus of continuity up to the boundary for the physical solution obtained in Theorem \ref{main}, in the case the regularity of the boundary datum does not overcome a threshold we are going to describe. Let  $\omega_g$ be a concave modulus of continuity for $g$:
\begin{equation}\label{mod.g}
\sup_{(x_0,t_0)\in \partial_p\Omega_T}\ \osc_{\overline{Q_r}(x_0,t_0)\cap\partial_p\Omega_T} g \leq \omega_g(r). 
\end{equation}
Given a point $(x_0,t_0) \in \R^{n+1}$ and a radius $r>0$, $Q_r(x_0,t_0)$ is the standard (symmetric) parabolic cylinder
\begin{equation*}
 Q_r(x_0,t_0):=B_r(x_0)\times (t_0-r^p,t_0+r^p);
\end{equation*}
$\overline{Q_r}(x_0,t_0) $ is its closure and, for a constant $\varrho >0$, $Q_r^{\varrho}(x_0,t_0)$ is the stretched cylinder
\[
Q_r^{\varrho}(x_0,t_0):=B_r(x_0)\times (t_0-\varrho^{2-p}r^p,t_0+\varrho^{2-p}r^p).
\]
Finally, let us introduce $\bar q\equiv \bar q(n,p)\geq2$ as
\begin{equation} \label{eq:lat q cond}
\bar q:=
\begin{cases}
  1+\displaystyle{\frac np} \qquad&  \text{for }p<n\,,\\[3pt]
  2\qquad & \text{for }p\geq n\,.
\end{cases}
\end{equation}
We are ready now to state
\begin{theorem}\label{boundary.continuity}
Let $u$ be the physical solution of Theorem \ref{main}, and let $(x_0,t_0)\in\partial_p\Omega_T$ and $R_0\in(0,r_\Omega]$ be fixed ($r_\Omega$ will be introduced in \eqref{eq:lateral density}). Then for every 
\begin{equation}\label{beta}
\alpha\in \bigg(0,\frac1{p'\bar q}\bigg)
\end{equation}
 there exist constants $\vartheta,\lambda_0,\tilde\delta$ depending on $\alpha$ and the data, such that if we set
\begin{equation} \label{choice.modulus.boundary}
\omega (r) =\frac1\vartheta\biggl[\frac{1}{\log\big(\log(\frac{\lambda_0R_0}{r})\big)}\biggr]^\alpha
\end{equation}
for $r\in(0,R_0]$, and we suppose that
\begin{equation}\label{ass.ug}
\osc_{Q_{R_0}^{\tilde \delta \omega_0}(x_0,t_0)\cap\Omega_T} u\leq \omega_0\qquad \text{and}\qquad  \omega_g\big((r/R_0)^{1-\gamma}\big)\leq M \omega(r)
\end{equation}
hold for some $\omega_0>0,M>0$ and $\gamma\in(0,1)$, then 
\begin{equation}\label{modulus.u}
 \osc_{Q_r(x_0,t_0)\cap\Omega_T} u \leq c\,\omega(r)
 \end{equation}
for all $r \in (0,R_0]$, with $c$ depending on $\gamma,M,R_0, \omega_0$ and the data.
\end{theorem}
The previous natural result tells that once the boundary datum is more regular than the solution, even in the case of smooth $g$, then the solution still has modulus of continuity $\omega$.  Clearly, a H\"older continuous function $g$ is an example of boundary datum satisfying \eqref{ass.ug}$_2$.

\subsection{Main assumptions and the concept of solution}\label{assump}

Throughout the paper, $\Omega$  is assumed to satisfy the following (standard in this context) outer density condition: there exist $\delta \in (0,1)$ and $r_\Omega>0$ such that, for $x_0 \in \partial \Omega$,
\begin{equation} \label{eq:lateral density}
| B_r(x_0) \cap \Omega| \leq (1-\delta) |B_r(x_0)|,  \qquad  \forall r \in (0,r_\Omega) .
\end{equation}
The function $\beta:\R\to\R$ is an {\em increasing} $C^1$-diffeomorphism satisfying the  bi-Lipschitz condition
\begin{equation}\label{bilipischitz}
\Lambda^{-1} |u-v| \leq |\beta(u) - \beta(v)| \leq \Lambda |u-v|\,, 
\end{equation}
included, as previously done in \cite{DiB2, Noche87}, to account for the thermal properties of the medium, which can change slightly with respect to the temperature.

The vector field $\mathcal{A}$ is measurable with respect to the first two variables and continuous
with respect to the last two, satisfying additionally the following standard growth, coercivity and monotonicity assumptions:
\begin{equation}\label{p.laplacian.assumptions}
\begin{split}
\qquad|\mathcal{A}(x,t,u,\xi)| \leq \Lambda |\xi|^{p-1} \,, \qquad \langle \mathcal{A}(x,t,u,\xi) , \xi \rangle \geq \Lambda^{-1} |\xi|^p   \,, \\[3pt]
\hspace{-2cm}\langle\mathcal{A}(x,t,u,\xi)-\mathcal{A}(x,t,u,\zeta),\xi-\zeta\rangle>0\,, \qquad\qquad
\end{split}
\end{equation}
for $p\geq2$, for almost every $(x,t) \in \Omega_T$ and for all $(u,\xi,\zeta) \in \mathbb{R} \times  \mathbb{R}^{2n}$, with $\zeta\neq\xi$, for a given constant $\Lambda \geq 1$. It will be useful for future reference to make explicit the modulus of continuity of $\mathcal A$ with respect to the last two variables; we suppose that there exist two concave functions $\omega_{\mathcal A,u},\omega_{\mathcal A,\xi}:(0,\infty)\to[0,1]$, such that $\lim_{\rho\searrow0}\omega_{\mathcal A,u}(\rho)=\lim_{\rho\searrow0}\omega_{\mathcal A,\xi}(\rho)=0$, and a function $K:[0,\infty)\times [0,\infty)\to[1,\infty)$, increasing separately in the two variables, such that
\begin{multline}\label{continuityA}
\sup_{(x,t)\in\Omega_T}|\mathcal A(x,t,u,\xi)- \mathcal A(x,t,v,\xi)|\\
\leq K(M,\tilde M)\,\Big[\omega_{\mathcal A,u}\big(|u-v|\big)+\omega_{\mathcal A,\xi}\big(|\xi-\zeta|\big)\Big]
\end{multline}
for all $(u,v,\xi,\zeta)\in \R^{2(n+1)}$ such that  $|u|+|v|\leq M$ and $|\xi|+|\zeta|\leq \tilde M$.

\begin{definition}\label{deff}
A {\em local weak solution} of equation \eqref{the general equation} is a pair $(u,v)$, with 
\[
v\in \beta(u)+H_a(\beta(u)),
\]
in the sense of graphs, such that
\[
u\in L^p_{\loc}(0,T;W^{1,p}_{\loc}(\Omega))\cap L^\infty_{\loc}(0,T;L^2_{\loc}(\Omega))=: V^{2,p}_{\loc}(\Omega_T)
\]
and the integral identity
\begin{equation}
\int_{\mathcal K}[ v\,\varphi](\cdot,\tau) \dx\biggr|_{\tau=t_1}^{t_2} +\int_{\mathcal K\times[t_1,t_2]} \big[-v\, \partial_t\varphi  +\langle \mathcal{A}(\cdot, \cdot, u,Du) ,D\varphi\rangle\big]\dx\dt= 0
\label{weakdef}
\end{equation}
holds for all $\mathcal K\Subset\Omega$ and almost every $t_1,t_2\in\R$ such that $[t_1,t_2]\Subset(0,T]$, and for every test function $\varphi \in L^p_{\loc}(0,T;W^{1,p}_0(\mathcal K))$ such that $\partial_t\varphi\in L^2(\mathcal K\times[t_1,t_2])$.
\end{definition}

\begin{remark}
Observe that also $v \in  L^\infty_{\loc}(0,T;L^2_{\loc}(\Omega))$ and that the test functions are in $C\left( [t_1,t_2]; L^2 (\mathcal K) \right)$, so every term in \eqref{weakdef} has a meaning.
\end{remark}

\subsection{Strategy of the proof}\label{strategy}
In order to perform a standard reduction of the oscillation, at least in cylinders centered on the lateral boundary, we shall consider {\rm three different alternatives}. The reduction of the oscillation in the interior has been proven in \cite{BKU}, while at the initial boundary it is a simple consequence of the logarithmic estimate of Lemma \ref{lemma: log lemma}. Let us give a brief and formal description of the structure of the proof. Consider equation \eqref{the general equation}; clearly we can suppose that the jump is met by the values of the solution in the cylinder considered, otherwise solutions are continuous since they solve $p$-Laplacian type equations with continuous Cauchy-Dirichlet data. The proof consists in the separate analysis of three alternatives.

Our first alternative \eqref{eq:b2 lat} states that the jump is far to the supremum of $u$ on the cylinder. In this case, we can reduce the supremum remaining ``above'' the jump, and here the equation behaves like the $p$-Laplace equation.

The second alternative \eqref{eq:b1 lat} instead means that we are considering the case where the jump is close to the supremum of $u$, and thus it is really influencing the behaviour of the solution.  In this case, we set two further alternatives, \eqref{lat alt.1} and \eqref{lat alt.2}: the latter describes the case where the solution has low energy levels close to the jump for all times (notice the relation between the condition appearing therein and the left-hand side of the energy estimate in \eqref{caccioppoli estimate}). Here the equation is still very similar to the $p$-Laplace equation and indeed we reduce the oscillation in $p$-Laplacian type cylinders. If this is not the case, that is if the worst case scenario \eqref{lat alt.1} happens, solutions are less regular. \eqref{lat alt.1} encodes the fact that the solution has a high peak of energy close to the jump; in this case, the presence of the jump is significant and therefore the geometry employed must rebalance the further degeneracy it produces.

The implementation of what is described above is quite technical, primarily due to the fact that, as is usual in degenerate evolutionary problems, time scales must depend on the solution itself. We need to define three different time scales to tackle the three different scenarios, and these are not trivial already in the non-degenerate case $p=2$.  Moreover, we have to introduce the exponentially small (in terms of the oscillation in the cylinder we are considering) quantity $\widetilde\omega$ in \eqref{tilde.omega} and this explicitly reflects in the $\log-\log$ modulus of continuity we obtain.

\section{Preparatory material}

\subsection{Approximation of the problem}\label{sec: approx}
Let $\rho_\eps$ be the standard symmetric, positive, one dimensional mollifier, supported in $(-\eps,\eps)$, obtained via rescaling of $\rho\in C^\infty_c(-1,1)$. We set
\begin{equation}\label{regularized}
H_{a,\eps}(s) := (\rho_\eps \ast H_a)(s),\quad s\in\R,
\end{equation}
and observe that $H_{a,\eps}$ is smooth and
\begin{equation}\label{support.H}
{\rm supp}\, H_{a,\eps}'\subset (a-\eps,a+\eps),\qquad \int_{\R}H_{a,\eps}'(v)\,dv=1. 
\end{equation}
Those will be the unique properties of $H_{a,\eps}$ we will use in the proofs of Section \ref{sec.three} (actually, we use the fact that the integral is bounded from above by one). Let $u_\eps$ solve the approximate Cauchy-Dirichlet problem
\begin{equation*}
\begin{cases}
\partial_t \big[\beta(u_\eps)+ H_{a,\eps}(\beta (u_\eps))\big] - \mathrm{div}\, \mathcal{A}(x,t,u_\eps,Du_\eps) = 0\quad&\text{in $\Omega_T$},\\[5pt] 
u_\eps= g&\text{on $\partial_p\Omega_T$}.
\end{cases}
\end{equation*}
Setting
\begin{equation}\label{w}
w_\eps := \beta(u_\eps)\,, \qquad w_0:=\beta(g),
\end{equation}
we arrive at the regularized Cauchy-Dirichlet problem
\begin{equation}\label{Approx.CD.reg}
\begin{cases}
\partial_t w_\eps -  \mathrm{div}\, \bar{\mathcal{A}}(x,t,w_\eps,Dw_\eps) = -  \partial_t H_{a,\eps}(w_\eps)\quad&\text{in $\Omega_T$},\\[5pt] 
w_\eps= w_0&\text{on $\partial_p\Omega_T$},
\end{cases}
\end{equation}
where
\begin{equation} \label{eq:mollified vf}
\bar{\mathcal{A}}(x,t,\mu,\xi):= \mathcal{A}\big(x,t,\beta^{-1}(\mu), [\beta^\prime(\beta^{-1}(\mu))]^{-1}\xi\big),
\end{equation}
for a.e. $(x,t) \in \Omega_T$, $(\mu,\xi) \in \R \times  \R^n$. Observe that the growth and ellipticity bounds for $\bar{\mathcal{A}}$ are inherited from $\mathcal{A}$ and from the two-sided bound for 
$\beta'$: indeed
\begin{equation}\label{tilde inf}
|\bar{ \mathcal{A}}(x,t,\mu,\xi)| \leq \Lambda^p |\xi|^{p-1} \,, \qquad \langle \bar{\mathcal{A}}(x,t,\mu,\xi) , \xi \rangle \geq \Lambda^{-p}
|\xi|^p
\end{equation}
hold for almost every $(x,t) \in \Omega_T$ and for all $(\mu,\xi) \in \R \times  \R^n$. Moreover, 
$\bar{\mathcal{A}}$ is clearly continuous with respect to the last two variables since $\beta$ is a $C^1$-diffeomorphism. 

By standard regularity theory for degenerate parabolic equations, see~\cite{DiBe93, KMWolff, Urba08}, we have that the solution $w_\eps$ of \eqref{Approx.CD.reg}$_1$ is H\"older 
continuous since $\beta(u_\eps)+ H_{a,\eps}(\beta (u_\eps))$ is now a diffeomorphism. This kind of regularity depends however on the regularization and as such it will deteriorate as $\eps \downarrow 0$. Nonetheless, we may assume that the solution of the regularized equation is continuous having, in particular, pointwise values. Sometimes we will use the compact notation
\begin{equation}\label{Acca}
\mathcal H(s):=s+ H_{a,\eps}(s). 
\end{equation}

\subsection{Scaling of the equation}\label{rescaling}
It will be useful later on to rescale the solution of \eqref{Approx.CD.reg} in the following way:  define, for $\lambda\geq1$, 
\begin{equation*} 
\hat v(y,\tau) :=  \frac{w_\eps(y,t_0 +  \lambda^{2-p}(\tau-t_0))}{\lambda}
\end{equation*}
in $E=\Omega\times(t_0-\lambda^{p-2}t_0,t_0+\lambda^{p-2}(T-t_0)]$. If we set
\[
\hat H(s):=\frac{H_{a,\eps}(\lambda s)}\lambda=\frac{H_{a/\lambda,\eps/\lambda}(s)}\lambda, \qquad \hat g(y,\tau):=\frac{ w_0(y,t_0 +  \lambda^{2-p} (\tau-t_0))}{\lambda},
\]
it is then easy to see that $\hat v$ solves the Cauchy-Dirichlet problem
\begin{equation}\label{Approx.CD.reg2}
\begin{cases}
\partial_\tau \hat v-{\rm div}_{y}\,\hat{\mathcal A}(y,\tau,\hat v,D\hat v)=-\partial_\tau\hat H(\hat v) \quad&\text{in $E$},\\[5pt] 
\hat v= \hat g&\text{on $\partial_p E$},
\end{cases}
\end{equation}
with $\hat{\mathcal A}(y,\tau,\mu,\xi):=\bar{\mathcal A}(y,\tau,\lambda \mu,\lambda\xi)/\lambda^{p-1}$ having the same structural properties as $\mathcal A$. Note that in particular we have
\[
{\rm supp}\,\hat H'\subset \left( \frac{a-\eps}\lambda,\frac{a+\eps}\lambda \right)\qquad\text{and}\qquad \int_{\R}\hat H'(\sigma)\,d\sigma\leq 1.
\]

\subsection{Sobolev's inequalities}\label{Sobo.paragraph}
We recall here, in a unified and slightly formal setting, some parabolic Sobolev-type inequalities that will be useful in the rest of the paper. To start with, we recall that we can denote the Sobolev conjugate exponent of $p$ as $p^*=\kappa p$, where 
\begin{equation} \label{eq:kappa}
\kappa := \begin{cases}
  \frac{n}{n-p} &\qquad  \text{for }p<n\,,   \\[3pt]
  \mbox{any number $>1$} &\qquad \text{for }p=n\,, \\[3pt]
   +\infty &\qquad\text{for } p>n\,.
\end{cases}
\end{equation}
For a ball $B$ in $\R^n$ and an interval $\Gamma$ of $\R$, we consider functions
\begin{multline*}
w\in L^p(\Gamma; W^{1,p}(B))\cap L^\infty(\Gamma; L^2(B))\\
\text{and}\qquad \phi\in C^\infty(B\times\Gamma),\quad \phi(\cdot,\tau)\in C^{\infty}_c(B) \qquad \text{for all $\tau\in\Gamma$;} 
\end{multline*}
applying H\"older's inequality with respect to the time variable with conjugate exponents $\kappa,\kappa'$, and afterwards the standard Sobolev's inequality slice-wise for functions in $W^{1,p}_0(B)$, we infer
\begin{multline}\label{Sobolev}
\mean{B\times \Gamma} w^{2(1-1/\kappa)  + p } \phi^{p(2-1/\kappa)}\dx \dt   \\
\leq  
c(n,p)\, |B|^{p/n} |\Gamma|^{1-1/\kappa} \biggl[ \frac{1}{|\Gamma|} \sup_{\tau\in \Gamma} \mean{B}  [w^2\phi^p](\cdot,\tau) \dx \biggr]^{1-1/\kappa} \mean{B\times\Gamma} {|D(w\phi)|}^p \dx \dt.
\end{multline}
From now on, we shall make use of the formal agreement that when $\kappa=\infty$, then $1/\kappa=0$, $\kappa/(\kappa-1)=1$ and 
\begin{equation}\label{agree}
\Big[\mean{B} (w\phi)^{\kappa p} \dx \Big]^{1/\kappa}:=\|w\phi\|_{L^\infty(B)}^p;
\end{equation}
note that in this case there is no necessity to apply H\"older's inequality.

\vs

Finally,  once chosen a number $\alpha$ as in \eqref{beta} and, setting $q=\frac{1}{p'\alpha}>\bar q$, with $\bar q$ defined in \eqref{eq:lat q cond}, we fix $\kappa\equiv \kappa(q)\equiv \kappa(p,\alpha)$, in the case $p=n$, as
\begin{equation*}
\kappa=\frac q{q-2}>1;
\end{equation*}
in the rest of the paper we shall implicitly keep $\kappa$ fixed with this value. This, in view of the fact that the lower bound  for $q$ satisfies the (formal when $\kappa = \infty$) relation $\bar q=1+\kappa/(\kappa-1)$, ensures that 
\begin{equation}\label{kappaq}
 \Big(1-\frac1q\Big)\Big(2-\frac1\kappa\Big)>1.
\end{equation}

\subsection{Notation}\label{Notation}
Our notation will be mostly self-explanatory; we mention here some noticeable facts. We shall follow the usual convention of denoting by $c$ a generic constant {\em always greater than or equal to one} that may vary from line to line; constants we shall need to recall will be denoted with special symbols, such as $\tilde c, c_\ast,c_\ell$ or the like. Dependencies of constants will be emphasized between parentheses: $c(n,p,\Lambda)$ will mean that $c$ depends only on $n,p,\Lambda$; often dependencies will be shown right after displays. By saying that a constant depends on the data, we mean that it depends on $n,p,\Lambda,\delta$.

By parabolic boundary of a cylinder $\mathcal K:=C\times \Gamma$, we shall mean $\partial_p\mathcal K:=(\overline C\times\{\inf \Gamma\})\cup (\partial C\times \Gamma)$. Its lateral boundary will be denoted as $\partial_{\rm lat}\mathcal K:=\partial C\times \Gamma$ and its initial boundary $\overline C\times\{\inf \Gamma\}$ will be $\partial_{\rm ini}\mathcal K$. We denote by ${(f)}_A$ the averaged integral
\begin{equation*}
      {(f)}_{A}:=\mean{A} f(\xi)\, d\xi := \frac{1}{|A|} \int_{A} f(\xi)\, d\xi,
\end{equation*}
where $A \in \R^k$ is a measurable set with $0<|A|<\infty$ and $f:A \to \R^m$ an integrable map, with $k,m \ge 1$. Finally we stress that with the statement ``a vector field with the same structure as $\mathcal A$'' (or ``structurally similar to $\mathcal A$'', or expressions alike) we shall mean that the vector field will satisfy \eqref{p.laplacian.assumptions}, eventually with $\Lambda$ replaced by a constant depending only on $n,p,\Lambda$, and continuous with respect to the last two variables. $\N$ is the set $\{1,2,\dots\}$, while $\N_0:=\N\cup\{0\}$.

\section{Reducing the oscillation at the boundary}\label{sec.three}

In this section we shall consider a function $v$ solving 
\begin{equation} \label{the equation for v.bnd}
\begin{cases}
\partial_t v - \mathrm{div} \, \widetilde{\mathcal{A}}(x,t,v,Dv)  = -\partial_t H_{b,\eps}(v)\qquad &\text{in $\Omega_T$,}\\[3mm]
v=\tilde g&\text{on $\partial_p\Omega_T$,}
\end{cases}
\end{equation}
with the Cauchy-Dirichlet datum $\tilde g$ being a uniformly continuous function and $\wt{\mathcal A}$ satisfying \eqref{p.laplacian.assumptions}$_{1,2}$. By regularity theory for evolutionary $p$-Laplace type equations, see \cite{DiBe93, Urba08}, we actually have that the solution $v$ is continuous
up to the boundary since $\sigma\mapsto\sigma+H_{b,\eps}(\sigma)$ is a diffeomorphism for $\eps>0$ fixed. Later on we shall take as $v$ the function $w_\eps$ appearing in \eqref{w}, conveniently rescaled (and this explains the fact that the jump happens at $s=b\neq a$), and as $\tilde g$ the boundary datum $w_0$, also rescaled. 

\vspace{3mm}

As the first result we have the following Caccioppoli's inequality at the boundary.
\begin{lemma}\label{caccioppoli lemma}
Let $v$ be a solution to \eqref{the equation for v.bnd} and let $Q=B \times \Gamma$ be a cylinder such that $Q\cap\partial_p\Omega_T\neq \emptyset$. Then there exists a constant $c$ depending on $p$ and $\Lambda$ such that
\begin{multline}\label{caccioppoli estimate}
 \sup_{\tau\in \Gamma\cap(0,T)}  \frac{1}{|\Gamma\cap(0,T)|} \mean{B\cap\Omega} \Bigl[\int_k^v H_{b,\eps}'(\xi)  (\xi-k)_+ \, d\xi  \, \phi^p\Bigr](\cdot,\tau) \dx\\
  + \sup_{\tau\in \Gamma\cap(0,T)}  \frac{1}{|\Gamma\cap(0,T)|} \mean{B\cap\Omega}\big[(v-k)_+^2 \phi^p\bigr](\cdot,\tau) \dx + \mean{Q\cap\Omega_T} \big| D (v-k)_+ \phi \big|^p \dx\dt\\
 \leq c \, \mean{Q\cap\Omega_T}\Big[ (v-k)_+^p |D\phi|^p + (v-k)_+^2 \left( \partial_t \phi^p\right)_+  \Big]\dx \dt \\
+ c \,   \mean{Q\cap\Omega_T}    \int_k^v H_{b,\eps}'(\xi)  (\xi-k)_+ \, d\xi  \left( \partial_t \phi^p\right)_+ \dx \dt
\end{multline}
for any $k>\sup_{\overline Q\cap\partial_p\Omega_T}\tilde g$ and any test function $\phi\in C^\infty(Q)$ vanishing on $\partial_p Q$.
\end{lemma}

\begin{proof}
In order to get \eqref{caccioppoli estimate}, we test the local weak formulation of \eqref{the equation for v.bnd} with $\varphi=(v-k)_+ \phi^p$; notice that $\varphi$ has a compact support in $\Omega_T$, since $v$ is continuous up to the boundary as it solves the regularized equation. The calculations are now standard and we refer to \cite[Lemma 2.1]{BKU} or \cite{DiB2}. 
\end{proof}
\begin{remark}\label{rem.sobolev}
Note that it makes sense to apply the Sobolev's inequality of \eqref{Sobolev} to functions of the form $\varphi:=(v-k)_+\phi$, $\phi\in C^\infty_c(B), k\in\R$ large as in Lemma \ref{caccioppoli lemma}, on balls centered on the lateral boundary of $\Omega_T$, just setting $\varphi\equiv 0$ outside $\Omega_T$. In view of \eqref{eq:lateral density}, taking averages in \eqref{Sobolev} with respect to $B$ is equivalent to taking them with respect to $B\cap\Omega$, so there will not be any possible misunderstanding. Another occurrence when we shall apply Sobolev's inequality \eqref{Sobolev} is when $\phi(\cdot,\tau)=0$ in $\R^n\setminus\Omega$, for almost every $\tau\in\Gamma$; in this case, we have
\[
|B\cap\{\phi(\cdot,\tau)=0\}|\geq \delta|B|\qquad\text{for a.e. $\tau$}
\]
by our density assumption \eqref{eq:lateral density}, and again a classic Sobolev-type inequality, see \cite[Theorem 1, p. 189]{EvansGariepy}, leads to \eqref{Sobolev}, with the constant also depending on $\delta$. Also a Poincar\'e's inequality is available in this case (see for instance \eqref{Po.boundary}). 
\vs

\end{remark}

\subsection{Reducing the oscillation at the lateral boundary}\label{red.lateral}
 Assume now that $(x_0,t_0)\in \partial_{\rm lat}\Omega_T$ and recall that $\Omega$ satisfies the outer density condition \eqref{eq:lateral density} with parameters $\delta \in (0,1)$ and $r_\Omega>0$. 
Let $\omega\in(0,1]$ and define the following auxiliary number, for $\eps_1\in(0,1)$ to be fixed later:
\begin{equation}\label{tilde.omega}
\widetilde{\omega} = \eps_1\omega \exp\left( - [\eps_1 \omega]^{- p' q}\right) <\frac12\eps_1\omega<\omega.
\end{equation}
We shall need to work with the two time scales $T^1:=\left[ \eps_1 \omega \right]^{2-p} r^p$, $T^3:=\widetilde\omega^{1-p}r^p$
in order to handle the degeneracy given by the jump. Moreover we shall also need the scale $T^2:= [\eps_2\widetilde\omega]^{2-p} r^p$, $\eps_2\in(0,1)$, when away from the jump, i.e., when dealing with the degeneracy given only by the $p$-Laplacian operator, see Paragraph \ref{uf}. We shall moreover always consider $\eps_1\leq \eps_2^{p-2}$, see \eqref{eps1}; in view of this, \eqref{tilde.omega}, and the trivial fact that $\widetilde\omega\leq \eps_1$, we have
\begin{equation}\label{time.partialorder}
 T^1=\left[ \eps_1 \omega \right]^{2-p} r^p\leq \widetilde\omega ^{2-p} r^p\leq T^2=\eps_2^{2-p}\widetilde\omega^{1+(1-p)} r^p\leq \widetilde\omega^{1-p}r^p=T^3\,. 
\end{equation}
We also define for $\sigma > 0$ the cylinders 
\[
\sigma Q^i:=\big(B_{\sigma r}(x_0) \times\left(t_0-\sigma T^i,t_0\right) \big) \cap \Omega_T, \qquad i=1,2,3.
\]
Note that clearly $Q^1\subset Q^2\subset Q^3$. 

\vs

From now on we shall write
\[
\mu^+ := \sup_{Q^3} v, \qquad \mu^- := \inf_{Q^3} v.
\]
We further assume that
\begin{equation}\label{tri.assumption}
b \in [\mu^-,\mu^+]
\end{equation}
and
\begin{equation} \label{eq:lat init 1}
\sup_{\overline Q^3\cap \partial_p \Omega_T } \tilde g \leq \mu^+-\frac{\omega}{8}\,,\qquad\qquad \eps \leq \frac{\widetilde{\omega}}{2}\,.
\end{equation}
We consider two cases: either the jump is close to the supremum of $v$
\begin{equation} \label{eq:b2 lat}
b \leq\mu^+ - 2\widetilde{\omega},\tag{Alt.\,1}
\end{equation}
or this does not hold:
\begin{equation} \label{eq:b1 lat}
b >\mu^+- 2\widetilde{\omega}\,.\tag{Alt.\,2}
\end{equation}
In the case of~\eqref{eq:b1 lat}, we consider the further two alternatives: either 
\begin{equation}\label{lat alt.1}
\sup_{\max\{0,t_0-\frac14 T^1\}<t<t_0} \mean{B_{r/4}\cap\Omega} \int_{\mu^+- 3{\widetilde\omega}}^{v(\cdot,t)}H_{b,\eps}'(\xi)  \, d\xi \dx  > \eps_3^{-1}\left[ \eps_1 \omega \right]^{q}\tag{Alt.\,2.1}
\end{equation}
is in force or the converse inequality
\begin{equation}\label{lat alt.2}
\sup_{\max\{0,t_0-\frac14 T^1\}<t<t_0} \mean{B_{r/4}\cap\Omega}\int_{\mu^+- 3{\widetilde\omega}}^{v(\cdot,t)} H_{b,\eps}'(\xi)  \, d\xi \dx  \leq  \eps_3^{-1}\left[ \eps_1 \omega \right]^{q} \tag{Alt.\,2.2}
\end{equation}
holds, where $q$ satisfies~\eqref{eq:lat q cond} and $\eps_3 \in (0,1)$ will be chosen later. Note that it would be equivalent (see  \eqref{eq:b2 lat} and \eqref{eq:lat init 1}$_2$ and consider also \eqref{support.H}) to put as the lower bound in the integral of $H_{b,\eps}'$ the point $b-\eps$; we keep this choice also to meet the formal explanation in Paragraph \ref{strategy}.
\subsubsection{Strategy of the proof revisited}
There are three free parameters $\eps_1, \eps_2,\eps_3$ appearing above. The strategy for choosing them is to first fix $\eps_2$ in the case~\eqref{eq:b2 lat}; this choice is independent of $\eps_1$ and $\eps_3$. We subsequently fix  $\eps_3$ in the analysis of~\eqref{eq:b1 lat} and~\eqref{lat alt.1}, see \eqref{eps3}, independently of $\eps_1$ and $\eps_2$, and finally, $\eps_1$ is chosen to depend on the data and $\eps_2,\eps_3$ while analyzing the case~\eqref{eq:b1 lat} and~\eqref{lat alt.2} (see \eqref{eps1}).  

\vs

\begin{lemma} \label{lemma:lat dens}
Suppose that $v$ is a weak solution to~\eqref{the equation for v.bnd} satisfying~\eqref{tri.assumption}, \eqref{eq:lat init 1} and suppose that $\eps_1,\eps_2$ are small enough $(\eps_1,\eps_2\leq 2^{-10})$. Then there is a constant $c_\ell \equiv c_\ell(n,p,\Lambda,\delta)\geq1$ such that the following holds:
\begin{itemize}
\item if $v$ satisfies the first alternative \eqref{eq:b2 lat}, then
\begin{equation} \label{eq:lat dens 3}
\frac{|\frac{1}{8}Q^2 \cap \{ v > \mu^+  - 2\eps_2\widetilde\omega\}|}{|\frac{1}{8}Q^2|} \leq  \frac{c_\ell}{[ \log(1/ \eps_2)]^{1/p'}};
\end{equation}
\item if $v$ satisfies the second alternative \eqref{eq:b1 lat}, then
\begin{equation} \label{eq:lat dens 2}
\frac{|\frac{1}{2}Q^3 \cap \{ v > \mu^+  - 8\widetilde{\omega}\}|}{|\frac{1}{2}Q^3|} \leq c_\ell [\eps_1 \omega]^q;
\end{equation}
\item if $v$ satisfies the second alternative \eqref{eq:b1 lat} and also \eqref{lat alt.2}, then
\begin{equation} \label{eq:lat dens 1}
\frac{|\frac{1}{8}Q^1 \cap \{ v > \mu^+  - 2\eps_1 \omega\}|}{|\frac{1}{8}Q^1|} \leq  \frac{c_\ell \eps_3^{-1/p}}{[ \log (1/\eps_1)]^{1/p'}}.
\end{equation}
\end{itemize}
\end{lemma}

\begin{proof}
Let us first prove~\eqref{eq:lat dens 1}. We define
\begin{equation}\label{levsfunctions}
k_j := \mu^+ - 2^{-j}  \omega , \qquad w_j:= (v-k_j )_+, \qquad \widehat w_j := \min\{w_j,k_{j+1}- k_j\}, 
\end{equation}
for all $3\leq j\leq \bar \jmath$, where $\bar \jmath$ is the integer satisfying
\begin{equation}\label{epsj}
2^{-(\bar\jmath+2)}<2\eps_1\leq 2^{-(\bar\jmath+1)}. 
\end{equation}
By~\eqref{eq:lat init 1}$_1$ we have that for all $j\geq3$
\begin{equation}\label{bound.kj}
k_j  \geq \mu^+ - \frac{\omega}{8}  \geq\sup_{Q^3\cap \partial_p \Omega_T } \tilde g \geq\sup_{Q^1\cap \partial_p \Omega_T} \tilde g;
\end{equation}
therefore, $w_j(\cdot,t)$ vanishes in a neighborhood of $\partial \Omega $ for every $t\in(t_0-T^1,t_0)$. Thus we may extend it to be zero outside of $\Omega$ in such a way that 
$$\hat w_j \in L^p(t_0-T^1,t_0;W^{1,p}(B_r(x_0))).$$
 The density condition~\eqref{eq:lateral density} readily implies that
\begin{equation}\label{density}
| B_{r/8}(x_0) \cap \{ \widehat w_j (\cdot,t) = 0 \} | \geq \delta |B_{r/8}(x_0)|
\end{equation}
for all $t\in (t_0-T^1,t_0)$. Using this condition we have by the standard application of the Poincar\'e's inequality that
\begin{equation} \label{eq:Sobobound} 
\int_{B_{r/8}} \widehat w_j(\cdot,t) \dx \leq c(n,\delta) \,r \int_{B_{r/8}} |D\widehat w_j(\cdot,t)|\dx
\end{equation}
for every $t\in (t_0-T^1,t_0)$. Now we integrate the previous inequality over $(t_0-\frac18T^1,t_0)$ and then estimate from below the left-hand side in the following way:
\begin{equation}\label{es.beloww}
\int_{\frac18Q^1} \widehat w_j \dx\dt  \geq  (k_{j+1}- k_{j}) \big|{\textstyle\frac18}Q^1 \cap \{v \geq k_{j+1}\}\big| =2^{-(j+1)} \omega \big|{\textstyle\frac18}Q^1 \cap \{v \geq k_{j+1}\}\big|. 
\end{equation}
By H\"older's inequality we bound from above
\begin{equation}\label{es.abbove}
 \int_{\frac18Q^1}   |D\widehat w_j| \dx \dt  \leq  \biggl[\int_{\frac18Q^1} |Dw_j|^p \dx \dt \biggr]^{1/p} \big|{\textstyle\frac18}Q^1 \cap \{ k_{j} < v < k_{j+1}\}\big|^{1/p'}\,. 
\end{equation}
Combining the above displays leads to
\begin{multline}  \label{lat basic 0}
\big|{\textstyle\frac18}Q^1  \cap \{v \geq k_{j+1}\}\big|  
\leq  c(n,\delta) |{\textstyle\frac18}Q^1|^{1/p} \biggl[r^{p} \big[2^{-j}  \omega\big]^{-p}\mean{\frac18Q^1} |D(v-k_j)_+|^p \dx \dt \biggr]^{1/p} \times\\   
\times \big|{\textstyle\frac18}Q^1 \cap \{ k_{j} < v < k_{j+1}\}\big|^{1/p'}.
\end{multline}
At this point we want to use the boundary Caccioppoli's inequality, Lemma~\ref{caccioppoli lemma}, with $Q=\frac14Q^1$, $k=k_j$, and $\phi\in C^\infty(\frac14Q^1)$ a standard cutoff function vanishing on the parabolic boundary with $0\leq\phi\leq 1, \phi\equiv 1$ on $\frac18Q^1$, and 
\[
|\partial_t\phi^p|\leq \frac{c(p)}{T^1},\quad |D\phi|\leq \frac{c}{r}.
\]
 Observing that by \eqref{epsj} we have for any $j\leq\bar\jmath$
\begin{equation*}
T^1=\left[ \eps_1 \omega \right]^{2-p} r^p\geq \big[ 2^{-(j+2)} \omega \big]^{2-p} r^p,
\end{equation*}
and after some simple algebraic manipulations we obtain 
\begin{multline}  \label{cacc lat 1}
 \mean{\frac18Q^1}   |D(v-k_j)_+|^p  \dx\dt \leq \frac{c}{r^p} \biggl[\mean{\frac14Q^1}\Bigl( {(v-k_j)_+^p} +   {(v-k_j)_+^2}[ 2^{-j} \omega ]^{p-2} \Bigr) \dx \dt\\
+  [ 2^{-j} \omega ]^{p-2}   \mean{\frac14Q^1} \int_{k_j}^v H_{b,\eps}'(\xi)  (\xi-k_j)_+ \, d\xi \dx\dt\biggr].
\end{multline}
Now we have to use \eqref{lat alt.2}: we can estimate using $(v-k_j)_+\leq 2^{-j}\omega$ and the facts that $b-\eps>\mu^+-3\wt\omega$ and $H'_{b,\eps}(\xi)=0$ whenever $\xi<b-\eps$
\[
\int_{k_j}^v H_{b,\eps}'(\xi)  (\xi-k_j)_+ \, d\xi\leq 2^{-j}\omega\,\int_{\mu^+-3\widetilde\omega}^v H_{b,\eps}'(\xi)  \, d\xi\,.
\]
Then, by \eqref{lat alt.2} we infer
\begin{align}\label{term.H.estimate}
[ 2^{-j} \omega ]^{p-2}  \mean{\frac14Q^1} \int_{k_j}^v &H_{b,\eps}'(\xi)  (\xi-k_j)_+ \, d\xi \dx\dt\notag\\
&\leq \big[ 2^{-j} \omega \big]^{p-1}   \sup_{\max\{0,t_0-\frac14 T^1\}<t<t_0}\mean{B_{r/4}\cap\Omega} \int_{\mu^+-3\widetilde\omega}^{v(\cdot,t)}  H_{b,\eps}'(\xi)  \, d\xi\,\dx \notag\\
&\leq \eps_3^{-1} \big[ 2^{-j} \omega \big]^{p} 
\end{align}
since $\eps_1\leq 2^{-j}$ by \eqref{epsj} and $q>1$. It follows by combining  \eqref{lat basic 0}, \eqref{cacc lat 1} and \eqref{term.H.estimate} that
\[
|{\textstyle\frac18}Q^1 \cap \{v \geq k_{j+1}\}|   \leq c\, \eps_3^{-1/p} |{\textstyle\frac18}Q^1|^{1/p}   |{\textstyle\frac18}Q^1 \cap \{ k_{j} < v < k_{j+1}\}|^{1/p'}.
\]
Taking the power $p'$ from both sides and then summing up for $j=3,\dots,\bar\jmath$ gives
\begin{align*}
(\bar\jmath-2)|{\textstyle\frac18}Q^1 \cap \{v \geq k_{\bar\jmath+1}\}|^{p'}  &\leq c\,\eps_3^{-p'/p} |{\textstyle\frac18}Q^1|^{p'/p} \sum_{j=3}^{\bar\jmath}|{\textstyle\frac18}Q^1 \cap \{ k_{j} < v < k_{j+1}\}|\\
&\leq c\,\eps_3^{-p'/p} |{\textstyle\frac18}Q^1|^{1/(p-1)+1} =c \,\eps_3^{-p'/p} \,|{\textstyle\frac18}Q^1|^{p'} 
\end{align*}
and hence, finally,
\begin{equation*}
\frac{|{\textstyle\frac18}Q^1 \cap \{v \geq \mu^+  - 2^{-(\bar\jmath+1)}\omega\}|}{|{\textstyle\frac18}Q^1|}  \leq \frac{c\, \eps_3^{-1/p}}{(\bar\jmath-2)^{1/p'}}
\end{equation*}
with $c$ depending on $n,p,\Lambda,\delta$. The result now follows easily, since $-5\geq(\log_2\eps_1)/2$ implies 
\[
\bar\jmath-2\geq -\log_2 \eps_1-5\geq -c\,\log \eps_1.
\]

\vs

We come to the proof of \eqref{eq:lat dens 2}. The levels $k_j$ and the functions $w_j,\hat w_j$ are defined exactly as in \eqref{levsfunctions} for $3\leq j\leq \bar\jmath$, but this time with $\bar \jmath$ being the integer satisfying
\[
2^{-(\bar\jmath+2)}\omega<8\eps_1 \exp\big( - [\eps_1 \omega]^{- p'q}\big)\omega=8\widetilde\omega\leq 2^{-(\bar\jmath+1)}\omega;
\]
again this yields $\widetilde\omega\leq 2^{-j}\omega$  for all $j\leq \bar\jmath$. Now we can proceed similarly as above, since \eqref{bound.kj} still clearly holds. Extending again $\hat w_j$ to zero outside $\Omega$ in such a way that $\hat w_j \in L^p(t_0-T^3,t_0;W^{1,p}(B_r(x_0)))$, we have \eqref{density} over $B_{r/2}$ for all $t\in (t_0-T^3,t_0)$ and hence \eqref{eq:Sobobound} in $(t_0-T^3,t_0)$. Integrating and again estimating from below the left-hand side as in \eqref{es.beloww} and the right-hand side as in \eqref{es.abbove} yields
\[
2^{-j} \omega \big|{\textstyle\frac12}Q^3 \cap \{v \geq k_{j+1}\}\big|   \leq  c\biggl[r^p\int_{\frac12Q^3} |Dw_j|^p \dx \dt \biggr]^{1/p} \big|{\textstyle\frac12}Q^3 \cap \{ k_{j} < v < k_{j+1}\}\big|^{1/p'}
\]
with $c\equiv c(n,\delta)$. Now, by the choice of $\bar\jmath$, we have for any $j\leq\bar\jmath$ that $T^3= \widetilde\omega^{1-p}r^p \geq  \big[2^{-j}\omega\big]^{1-p}r^p$. Thus the boundary Caccioppoli's inequality in this case takes the form
\begin{align*}
 \mean{\frac12Q^3}   |D(v-k_j)_+|^p  \dx\dt &\leq \frac{c}{r^p}\biggl[\mean{Q^3}\Bigl( {(v-k_j)_+^p} +  (v-k_j)_+^2[2^{-j}\omega]^{p-1} \Bigr) \dx \dt\notag\\
&\hspace{1.5cm}+  [2^{-j}\omega]^{p-1}  \mean{Q^3} \int_{k_j}^v  H_{b,\eps}'(\xi)  (\xi-k_j)_+ \, d\xi \dx\dt\biggr].\notag
\end{align*}
Now, recalling that $(v-k_j)_+\leq 2^{-j}\omega$, we simply estimate by \eqref{support.H}$_2$
\[
\int_{k_j}^v H_{b,\eps}'(\xi)  (\xi-k_j)_+ \, d\xi\leq (v-k_j)_+\int_{\R} H_{b,\eps}'(\xi)   \, d\xi\leq 2^{-j}\omega
\]
and this leads to
\[
 \mean{\frac12Q^3}   |D(v-k_j)_+|^p  \dx\dt  \leq \frac{c}{r^p}\big[2^{-j}\omega\big]^p.
\]
This is to say, the choice of the time scale $T^3$ is sufficient to rebalance the inequality. We obtain  again
\[
|{\textstyle\frac12}Q^3 \cap \{v \geq k_{j+1}\}|   \leq c\, |{\textstyle\frac12}Q^3|^{1/p}   |{\textstyle\frac12}Q^3 \cap \{ k_{j} < v < k_{j+1}\}|^{1/p'}.
\]
and as above, after summing up for $j=3,\dots,\bar\jmath$ gives
\[
(\bar\jmath-2)^{1/p'}|{\textstyle\frac12}Q^3 \cap \{v \geq \mu^+-8\widetilde\omega\}| \leq(\bar\jmath-2)^{1/p'}|{\textstyle\frac12}Q^3 \cap \{v \geq k_{\bar\jmath+1}\}| \leq c \,|{\textstyle\frac12}Q^3|.
\]
We again conclude by estimating
\[
2^{-(\bar\jmath+2)}\leq 8\eps_1\exp\big( - [\eps_1 \omega]^{- p'q}\big)\leq 2^{-7-[\eps_1 \omega]^{- p'q}}
\]
since $\eps_1\leq2^{-10}$ and thus
\[
\bar\jmath-2 \geq \,[\eps_1 \omega]^{- p' q}.
\]

\vs

We are left with~\eqref{eq:lat dens 3}.  Defining now
\[
k_j := \mu^+ - 2^{-j}  \widetilde\omega, 
\]
$j\leq\bar\jmath$, where $2^{-(\bar\jmath+2)}<2\eps_2\leq 2^{-(\bar\jmath+1)}$, we notice that the proof, which on the other hand follows closely that of~\eqref{eq:lat dens 1}, reduces to the proof for the standard evolutionary $p$-Laplacian, because the phase transition lies outside of the image of $w_j$: indeed $b+\eps\leq \mu^+-2\wt\omega+\frac{\wt\omega}{2}<k_j$ for $j\in \N_0$, as a consequence of~\eqref{eq:b2 lat} and~\eqref{eq:lat init 1}$_2$. Hence the singular term drops from the Caccioppoli's inequality and the time scale $T^2$ rebalances it as in the usual case: for details see, for example,~\cite{DiBe93,DiBeUrbaVesp04,Urba08} and the forthcoming \eqref{Cac.standard.u}.
\end{proof}

\subsubsection{The geometric setting}
Due to the three different cases we consider (and subsequently, with the three different time scales needed), we shall need to work with three families of shrinking cylinders and related cutoff functions.

\vspace{3mm}

Set, for $j\in\N_0$,
\begin{equation*}
\sigma_j := \frac{1}{16}\big(1+2^{-j}\big) ,\qquad \wt\sigma_j := \frac{1}{4}\big(1+2^{-j}\big),
\end{equation*}
and
\[
Q_j^i := \sigma_j Q^i=\big(B_j\times(t_0-T_j^i,t_0)\big)\cap\Omega_T, \quad i=1,2,\quad Q_j^3 := \wt \sigma_j Q^3=\big(\wt B_j\times(t_0-T_j^3,t_0)\big)\cap\Omega_T,
\]
where
\[
B_j:=B_{\sigma_jr}(x_0),\quad T_j^i:=\sigma_jT^i,\quad i=1,2,\quad \wt B_j:=B_{\wt\sigma_jr}(x_0),\quad T_j^3:=\wt\sigma_jT^3.
\]
Note that
\[
\frac18Q^1=Q_0^1\supset Q_j^1\stackrel{j\to\infty}{\longrightarrow} \frac1{16}Q^1,\qquad\frac18Q^2=Q_0^2\supset Q_j^2\stackrel{j\to\infty}{\longrightarrow} \frac1{16}Q^2,
\]
and
\[
\frac12Q^3=Q_0^3\supset Q_j^3\stackrel{j\to\infty}{\longrightarrow} \frac14Q^3.
\]
We will take, for $i=1,2,3$ and $j\in \N_0$, standard smooth cut-off functions $\phi_{i,j}$ such that $\phi_{i,j}$ vanishes on the parabolic boundary of $Q_j^i$; moreover we assume $0\leq \phi_{i,j} \leq 1$ and $\phi_{i,j}\equiv1$ on $Q_{j+1}^i$. Note that we may also require 
\[
|\partial_t\phi_{i,j}^p|\leq c(p)\frac{2^j}{T^i},\qquad |D\phi_{i,j}|\leq c\,\frac{2^j}{r}.
\]

\subsubsection{Occurrence of~\eqref{eq:b2 lat}} \label{uf}
Here we state that using~\eqref{eq:lat dens 3} it is possible to show that
\begin{equation}\label{bound.u.thirdcase}
\sup_{\frac1{16}Q^2} v \leq \mu^+ - \eps_2 \widetilde\omega \,, 
\end{equation}
provided we choose $\eps_2\equiv \eps_2(n,p,\Lambda,\delta)$ small enough. Indeed, the proof for the above fact reduces (more or less) to the analysis of the standard evolutionary $p$-Laplacian operator, because the phase transition lies outside of the support of the test functions; essentially, we follow the proof of \cite[Lemma 9.1, Chapter III]{DiBe93}, once having~\eqref{eq:lat dens 3} at hand. We sketch the proof for the convenience of the reader. 

Choose for $j\in\N_0$ the levels 
\begin{equation*}
k_j := \mu^+  - \big(1+2^{-j}\big) \eps_2\widetilde\omega,
\end{equation*}
and consider the Caccioppoli inequality, Lemma~\eqref{caccioppoli lemma}, with $Q=Q_j^2$, $k=k_j$, and $\phi=\phi_{2,j}$. Noting that $k_j\geq \mu^+  - \frac32\widetilde\omega\geq b+\eps$ and recalling that $T^2=[\eps_2\widetilde\omega]^{2-p}r^p$ we have 
\begin{multline}\label{Cac.standard.u}
\frac1{\min\{T_j^2,t_0\}}  \sup_{\max\{0,t_0-T_{j}^2\}<t<t_0}  \mean{B_{j}\cap\Omega}  \big[(v-k_j)_+^2\phi_{2,j}^p\big](\cdot,t) \dx\\
 + \mean{Q_{j}^2} \left|D(v-k_j)_+\phi_{2,j}\right|^p \dx\dt\\
\leq c \, \frac{2^{jp}}{r^p} \mean{Q_j^2}\Bigl((v-k_j)_+^p + (v-k_j)_+^2[\eps_2\widetilde\omega]^{p-2} \Bigr) \dx \dt\,,
\end{multline}
with $c\equiv c(n,p,\Lambda)$. Using 
\[
2^{-(j+1)}\eps_2\widetilde\omega\,\chi_{\{v> k_{j+1}\}}\leq (v-k_j)_+\leq  2\eps_2\widetilde\omega
\]
and Sobolev's inequality \eqref{Sobolev} (see also Remark \ref{rem.sobolev}) together with \eqref{Cac.standard.u}, we have for all $j\in\N_0$ 
\begin{align}\label{third.alt}
 &\big[2^{-(j+1)}\eps_2\widetilde\omega\big]^{2(1-1/\kappa)+p}A_{j+1}\leq \mean{Q_{j+1}^2} (v-k_j)_+^{2(1-1/\kappa)+p} \dx \dt  \notag \\
 &\leq c\, r^p \big[T^2\big]^{1-1/\kappa} \left[\frac1{\min\{T_j^2,t_0\}} \sup_{\max\{0,t_0-T_{j}^2\}<t<t_0}  \mean{B_{j}\cap\Omega}  \big[(v-k_j)_+^2\phi_{2,j}^p\big](\cdot,t) \dx\right]^{1-1/\kappa}\notag\\
 &\hspace{6cm}\times\mean{Q_{j}^2} \left|D(v-k_j)_+\phi_{2,j}\right|^p \dx\dt\notag\\
&\leq   c\, r^p \big[T^2\big]^{1-1/\kappa} \biggl[ \frac{2^{jp}}{r^p}    \mean{Q_j^2}\Bigl( (v-k_j)_+^p +  (v-k_j)_+^2[\eps_2\widetilde\omega]^{p-2} \Bigr) \dx \dt \biggr]^{2-1/\kappa} \notag\\
&\leq   c\, 2^{(2-1/\kappa)pj} r^{p+p(1-1/\kappa)-p(2-1/\kappa)}[\eps_2\widetilde\omega]^{(1-1/\kappa)(2-p)+p(2-1/\kappa)}\bar A_j^{2-1/\kappa} 
\end{align}
with 
\[
A_j := \mean{Q_j^2} \chi_{\{v>k_j\}} \dx \dt  =\frac{|Q_j^2\cap \{v > k_j\}|}{|Q_j^2|}\,.
\]
Thus $A_{j+1}\leq   c\, 2^{c(p,\kappa)j} A_j^{2-1/\kappa}$, with $c$ depending on $n,p,\Lambda,1/\kappa$. This yields \eqref{bound.u.thirdcase} in view of \eqref{eq:lat dens 3} and a standard hyper-geometric iteration lemma, provided $\eps_2$ is chosen small enough, in dependence of $n,p,\Lambda,\delta$ and $q$. Recall that $\kappa\equiv \kappa(q)$.

\subsubsection{Occurrence of~\eqref{eq:b1 lat} and~\eqref{lat alt.1}}
Set for $j\in\N_0$
\begin{equation*}
k_j := \mu^+ - 4(1+2^{-j})\widetilde\omega
\end{equation*}
and notice that $k_j<\mu^+ - 4\widetilde\omega$, which  together with~\eqref{eq:b1 lat} and \eqref{eq:lat init 1}$_2$ implies $b-\eps-k_j\geq\widetilde\omega$. Thus, using~\eqref{lat alt.1} we obtain 
\begin{align} \label{eq:2->1}
\sup_{\max\{0,t_0-T_{j+1}^3\}<t<t_0}  &\mean{\wt B_{j+1}\cap\Omega} \int_{k_j}^{v(\cdot,t)} H_{b,\eps}'(\xi)  (\xi-k_j)_+ \, d\xi \dx \notag\\
&\geq\sup_{\max\{0,t_0-\frac14T^3\}<t<t_0}  \mean{B_{r/4}\cap\Omega} \int_{b-\eps}^{v(\cdot,t)} H_{b,\eps}'(\xi)  (\xi-k_j)_+ \, d\xi \dx \notag\\
&\geq \widetilde\omega\sup_{\max\{0,t_0-\frac14T^1\}<t<t_0}  \mean{B_{r/4}\cap\Omega} \int_{\mu^+-3\wt\omega}^{v(\cdot,t)} H_{b,\eps}'(\xi)   \, d\xi \dx \notag\\
& > \eps_3^{-1}\widetilde\omega \left[ \eps_1 \omega \right]^{q}\,,
\end{align}
in view of \eqref{time.partialorder}. By Poincar\'e's inequality (see Remark \ref{rem.sobolev}) we have
\begin{equation}\label{Po.boundary}
 \mean{Q_{j+1}^3} (v-k_j)_+^p \dx \dt \leq c(n,p,\delta)\, r^p \mean{Q_{j+1}^3} |D(v-k_j)_+|^p \dx \dt,
\end{equation}
which together with \eqref{eq:2->1} and the Caccioppoli inequality with $Q=Q_j^3$, $k=k_j$, and $\phi=\phi_{3,j}$ yields 
\begin{align*}
& \notag  \frac{\eps_3^{-1} \widetilde\omega   \left[ \eps_1 \omega \right]^{q}}{T^3}  \mean{Q_{j+1}^3} (v-k_j)_+^p \dx \dt
\\  &  \leq  c\,\biggl[\frac{1}{\min\{T_{j+1}^3,t_0\}}  \sup_{\max\{0,t_0-T_{j+1}^3\}<t<t_0}  \mean{\wt B_{j+1}\cap\Omega} \int_{k_j}^{v(\cdot,t)} H_{b,\eps}'(\xi)  (\xi-k_j)_+ \, d\xi \dx\biggr]\times\\
&\hspace{7cm}\times\biggl[r^p \mean{Q_{j+1}^3} |D(v-k_j)_+|^p \dx \dt\biggr] \\
&  \leq  c\,2^{2pj}r^p\biggl[  \mean{Q_j^3}\Big( \frac{(v-k_j)_+^p}{r^p}+\frac{(v-k_j)_+^2}{T^3}\Big) \dx \dt\\
&\hspace{6cm}+ \mean{Q_j^3} \int_{k_j}^v H_{b,\eps}'(\xi)  \frac{(\xi-k_j)_+}{T^3} \, d\xi \dx\dt \biggr]^2.
\end{align*}
At this point, to bound both the left and the right-hand side, we use the following facts: first, we have
\[
2^{-(j-1)} \widetilde\omega\, \chi_{\{v>k_{j+1}\}}\leq (v-k_j)_+ \leq 8\widetilde\omega\,\chi_{\{v>k_j\}}\,;
\]
then, the definition of $T^3= \widetilde\omega^{1-p} r^p$ and also the fact that $\widetilde\omega\leq1$ yield
\[
2^{-pj}  \eps_3^{-1} \left[ \eps_1 \omega \right]^{q}  \mean{Q_{j+1}^3} \chi_{\{v>k_{j+1}\}}\dx\dt\\
\leq c\,2^{2pj}\,\biggl( \mean{Q_j^3}\chi_{\{v>k_j\}} \dx \dt \biggr)^2,
\]
with $c\equiv c(n,p,\Lambda,\delta)$. Denoting
\[
\bar A_j := \mean{Q_j^3} \chi_{\{v>k_j\}}\dx\dt= \frac{|Q_j^3 \cap \{v>k_j\}|}{|Q_j^3|}
\]
we hence finally have
\[
\bar A_{j+1} \leq   2^{3pj} \,\bar c\,\eps_3\left[ \eps_1 \omega  \right]^{-q} \bar A_{j}^{2}, 
\]
where the constant $\bar c$ depends on $n,p,\Lambda,\delta$, but it is independent of $\eps_1$.
Then, if 
\[
\bar A_0 \leq   \frac{\left[ \eps_1 \omega \right]^{q}}{\eps_3  2^{3p}\bar c}\,, 
\]
then the sequence $\{A_j\}$ becomes infinitesimal, in particular implying that
\begin{equation}\label{bound.u.firstcase}
\sup_{\frac14Q^3} v \leq \mu^+ - 4\widetilde\omega.  
\end{equation}
The above condition for $\bar A_0$ can be certainly guaranteed by taking 
\begin{equation}\label{eps3}
\eps_3 := \frac{1}{2^{3p} c_\ell \bar c},
\end{equation}
since Lemma~\ref{lemma:lat dens}, equation \eqref{eq:lat dens 2}, gives us exactly
\begin{equation*}
\frac{|Q_0^3 \cap \{ v > k_0\}|}{|Q_0^3|} \leq c_\ell {[\eps_1 \omega]}^q;
\end{equation*}
recall that we are assuming here~\eqref{eq:b1 lat}. Note carefully that now the parameter $\eps_3$ has been fixed as a parameter of $n,p,\Lambda,\delta$, but it is independent of $\eps_1$.

\subsubsection{Occurrence of~\eqref{eq:b1 lat} and~\eqref{lat alt.2}}
We set this time for $j\in\N_0$
\begin{equation*}
k_j := \mu^+  - \big(1+2^{-j}\big) \eps_1 \omega. 
\end{equation*}
Choosing $Q=Q_j^1, k=k_j$ and $\phi=\phi_{1,j}$ the Caccioppoli's estimate takes the form
\begin{align*}
  &\frac1{\min\{T_j^1,t_0\}}  \sup_{\max\{0,t_0-T_j^1\} <t<t_0}  \mean{B_{j}\cap\Omega}  \big((v-k_j)_+^2 \phi_{1,j}^p\big)(\cdot,t) \dx \\
   & \hspace{6cm}+ \mean{ Q_j^1} |D(v-k_j)_+\phi_{1,j}|^p \dx\dt\\
 &\quad\leq c \, 2^{pj}  \bigg[  \mean{Q_{j}^1}\left( \frac{(v-k_j)_+^p}{r^p} +  \frac{(v-k_j)_+^2}{T^1} \right) \dx \dt\\
 &\hspace{6cm}+  \frac{1}{T^1} \mean{Q_j^1} \int_{k_j}^v H_{b,\eps}'(\xi)  (\xi-k_j)_+ \, d\xi \dx\dt \bigg].
\end{align*}
Now using $(v-k_j)_+\leq2\eps_1\omega$,  H\"older's inequality and \eqref{lat alt.2} yields 
\begin{align*}
 &\mean{Q_j^1}  \int_{k_j}^v H_{b,\eps}'(\xi)  (\xi-k_j)_+ \, d\xi \dx\dt \notag\\
 &\leq 2\eps_1\omega \biggl(\mean{Q_j^1} \biggl[\int_{\mu^+- 3{\widetilde\omega}}^v H_{b,\eps}'(\xi)\,d\xi\biggr]^q \dx\dt\biggr)^{\frac1q} \biggl(\mean{Q_j^1} \chi_{\{v>k_j\}}\dx\dt\biggr)^{1-\frac1q}\notag\\
 &\leq 2\eps_1\omega \biggl(\sup_{\max\{0,t_0-T_j^1\} <t<t_0}\mean{B_{r/4}\cap\Omega} \int_{\mu^+- 3{\widetilde\omega}}^{v(\cdot,t)} H_{b,\eps}'(\xi)\,d\xi\dx\biggr)^{\frac1q} \biggl(\mean{Q_j^1} \chi_{\{v>k_j\}}\dx\dt\biggr)^{1-\frac1q}\notag\\
 &\leq c\,\eps_3^{-1/q} [\eps_1\omega]^2 \widetilde A_j^{1-\frac1q}\,,
\end{align*}
setting
\[
\widetilde A_j := \mean{Q_j^1} \chi_{\{v>k_j\}}\dx\dt=\frac{| Q_j^1\cap \{v > k_j\}|}{|Q_j^1|}\,,
\]
recalling that $q>1$ and $\int_\R H_{b,\eps}'\,d\xi\leq 1$. Also recall that $\eps_3$ is fixed and depends only on $n,p,\Lambda,\delta$.
Combining the two displays above and recalling that $T^1 = [\eps_1 \omega]^{2-p} r^p$, we obtain
\begin{multline*}
\frac1{\min\{T_j^1,t_0\}}  \sup_{\max\{0,t_0-T_j^1\} <t<t_0}  \mean{B_{j}\cap\Omega}  \big((v-k_j)_+^2 \phi_{1,j}^p\big)(\cdot,t) \dx   + \mean{ Q_j^1} |D(v-k_j)_+\phi_{1,j}|^p \dx\dt\\
\leq c \, 2^{pj}\frac{\left[\eps_1 \omega\right]^p}{r^p} \left(\widetilde A_j+\eps_3^{-1/q}\widetilde A_j^{1-\frac1q}\right)\leq c\,\eps_3^{-1/q} \, 2^{pj}  \frac{\left[\eps_1 \omega\right]^p}{r^p} \widetilde A_j^{1-1/q}\,.
\end{multline*}
To conclude, by Sobolev's inequality \eqref{Sobolev} with $\phi:=\phi_{1,j}$, $w=(v-k_j)_+$, $B=B_j\cap\Omega$ and $\Gamma=\left(\max\{0,t_0-T_j^1\},t_0\right)$ we infer
\begin{align*}
&\mean{Q_{j+1}^1} (v-k_j)_+^{2(1-1/\kappa)  + p } \dx \dt   \\
&    \leq  
c\, r^p \big[T^1\big]^{1-1/\kappa} \biggl[ \frac{1}{\min\{T_j^1,t_0\}} \sup_{\max\{0,t_0-T_{j}^1\}<t<t_0} \mean{B_j\cap\Omega}  \big((v-k_j)_+^2 \phi_{1,j}^p\big)(\cdot,t) \dx \biggr]^{1-1/\kappa} \times\notag\\
&\hspace{7cm}\times \mean{Q_j^1} {|D(v-k_j)_+\phi_{1,j}|}^p \dx \dt\notag\\
&  \leq  c\,\eps_3^{-(2-1/\kappa)/q} r^{p+p(1-1/\kappa)-p(2-1/\kappa)} [\eps_1 \omega ]^{(2-p)(1-1/\kappa)+p(2-1/\kappa)}\times\\
&\hspace{7cm}\times 2^{p(2-1/\kappa)j} \widetilde A_j^{(1 -1/q )(2-1/\kappa) } \,,\notag
\end{align*}
with $c$ depending on $n,p,\Lambda,q,\delta$. Estimating finally
\[
(v-k_j)_+ \geq 2^{-(j+1)}[\eps_1 \omega] \chi_{\{v>k_{j+1}\}}
\]
we conclude with
\[
\widetilde A_{j+1} \leq \tilde c\,\eps_3^{-(2-1/\kappa)/q}\,2^{4pj} \widetilde  A_j^{1+\zeta},
\]
where $\zeta:=(1 -1/q )(2-1/\kappa)-1>0$ by~\eqref{kappaq} and $\tilde c$ depends only on $n,p,\Lambda,\delta$ and $q$. 
Hence by choosing 
\begin{equation}\label{eps1}
\eps_1: =\min\left\{ \exp\left[-\left(c_\ell\,\tilde c^{1/\zeta}2^{4p/\zeta^2}  \eps_3^{-(1/p+(2-1/\kappa)/(\zeta q))}  \right)^{p'}\right],\eps_2^{p-2},\eps_2,2^{-10}\right\}\,,
\end{equation}
we get by~\eqref{eq:lat dens 1} that
\[
\widetilde A_0 \leq  \left[\tilde c\,\eps_3^{-(2-1/\kappa)/q}\right]^{-1/\zeta}2^{-4p/\zeta^2}
\]
and again a standard hyper-geometric iteration lemma ensures that
\begin{equation}\label{bound.u.secondcase}
\sup_{\frac1{16}Q^1} v \leq \mu^+ - \eps_1 \omega .  
\end{equation}
Note that, taking into account the fact that $\eps_2$ has already been fixed in Paragraph \eqref{uf} as constant depending on $n,p,\Lambda,\delta$ and $q$ and also $\eps_3$ has been fixed in \eqref{eps3} depending only on $n,p,\Lambda,\delta$, now also $\eps_1$ is fixed as a constant depending only on $n,p,\Lambda,\delta$ and $q$.

\subsubsection{Conclusion}\label{vartheta.fixed}
All in all, merging the three different alternatives that yield \eqref{bound.u.thirdcase}, \eqref{bound.u.firstcase} and \eqref{bound.u.secondcase}, then adding $-\inf_{\frac1{16}Q^1}v\leq-\mu^-$, we have proved that if $v$ is a solution to \eqref{the equation for v.bnd} and \eqref{eq:lat init 1} holds, then
\begin{equation}\label{allinall}
\osc_{\frac1{16}Q^1} v \leq \osc_{Q^3}v - \eps_1 \widetilde\omega \leq \osc_{Q^3}v - \eps_1^2\omega \exp\big(-[\eps_1\omega]^{-p'q}\big)\,.
\end{equation}
Indeed if $b$ satisfies \eqref{tri.assumption}, then \eqref{allinall} is what we proved on the previous pages. On the other hand, if $b\not \in [\inf_{Q^3} v, \sup_{Q^3} v]$, we are essentially in the same situation as described in Paragraph \ref{uf} and therefore also in this case \eqref{bound.u.thirdcase}, and hence \eqref{allinall}, holds. Note that if $b \not \in [\inf_{Q^3} v, \sup_{Q^3} v]$, then for $\eps$ small enough $v$ is a solution to the evolutionary $p$-Laplace equation, and the oscillation reduction follows in general by the well-known argument of DiBenedetto, see~\cite{DiBe93, Urba08}; however, referring also in this case to Paragraph \ref{uf} allows for a unitary treatment of these alternatives.

\begin{remark}\label{vartheta.fixed.2}
Note that in case 
\begin{equation}\label{red.osc.inf}
\inf_{\overline Q^3\cap \partial_p \Omega_T } \tilde g \geq \mu^- +\frac{\omega}{8}\,,\qquad\qquad \eps \leq \frac{\widetilde{\omega}}{2}\,, 
\end{equation}
holds in place of \eqref{eq:lat init 1}, then \eqref{allinall} still holds  since $-v$ solves an equation similar to \eqref{the equation for v.bnd} with boundary datum $-\tilde g$.
\end{remark}

\subsection{Reducing the oscillation at the initial boundary}
Let us take $x_0 \in \overline{ \Omega}$. Similarly to the previous Paragraph, here we denote, for some $\omega>0$
\begin{equation}\label{T.ini}
Q:= \big(B_r(x_0) \cap \Omega\big) \times (0,T^4), \qquad T^4 :=   \min\{\omega^{2-p}  r^p,T\}
\end{equation}
and we consider the function $v$ solving \eqref{the equation for v.bnd} with Cauchy-Dirichlet datum $\tilde g$. Let us remind the reader that the Caccioppoli's inequality of Lemma \ref{caccioppoli lemma} is valid for $v$ also in this case.  We can then follow the steps in~\cite[Chapter III, Section 11]{DiBe93} using time independent cut-off functions and we can reduce the problem to the analysis of the standard evolutionary $p$-Laplace equation; we briefly present the proof adapted to our setting. 

The next result is a standard ``Logarithmic Lemma", see for example the proof in~\cite[Chapter II]{DiBe93}. The assumption in \eqref{starrr} will be satisfied by imposing a proper condition between the solution and the initial trace $g(\cdot,0)$, see \eqref{bound.nontrivial.1}.
\begin{lemma}\label{lemma: log lemma}
Let $Q$ and $T^4$ be as in \eqref{T.ini}, and assume that $v \in C(\overline Q)$ solves~\eqref{the equation for v.bnd} in $Q$ and
\begin{equation}\label{starrr}
\sup _{B_r(x_0)\cap\Omega}v(\cdot,0) \leq \sup_Q v - \frac{\omega}{8}\,. 
\end{equation}
Then, for a constant $c$ depending on $n,p,\Lambda$, there holds
\begin{equation}\label{tritri}
\frac{\left| \big(B_{r/2} (x_0)\cap\Omega\big)\cap \left\{ v(\cdot , \tau ) \geq  \sup_Q v - {\theta}\, \omega/8 \right\} \right| }{| B_{r/2}  (x_0)\cap\Omega | } \leq  \, \frac{c}{ \log (1/{\theta}) } 
\end{equation}
whenever ${\theta} \in (0,1)$ and $\tau \in (0,T^4)$.
\end{lemma}
\begin{proof}
Denote in short $\widetilde{\mathcal{A}}(Dv):=\widetilde{\mathcal{A}}(x,t,v,Dv)$, $\hat B:=B_r(x_0)\cap\Omega$ and $\mathcal H$ as in \eqref{Acca}, with $b$ replacing $a$. Consider a time independent cut-off function $\phi \in C_c^\infty(B_r (x_0))$, $0\leq \phi \leq 1$, with $\phi \equiv 1$ in $B_{r/2}$, $\phi = 0$ on $\partial B_r (x_0)$, and $|D \phi | \leq c/r$. Take $k = \sup_Q v - \omega/8$ and define for ${\theta}\in (0,1/8]$ the function
\[
\varPsi(v) = \biggl[\log \left( \frac{\omega}{\omega(1+{\theta}) -8 (v-k)_+} \right)\biggr]_+\,.
\]
We have $\varPsi(v) \neq 0$ when $v>\sup_Q v-\omega(1-{\theta})/8=:v_- > \sup_Q v - \omega/8>\omega/2$ (note that if $\sup_Q v\leq 3\omega/4$ there is nothing to prove, since \eqref{tritri} would be trivial). Observe that we have
\[
\varPsi^\prime (v) = \chi_{\{ v > v_-\}} \frac{8}{\omega(1+{\theta}) - 8(v-k)_+ }.
\]
Testing formally the equation with $\eta = \varPsi^\prime (v) \varPsi(v) \phi^p\chi_{(-\infty, \tau)}(t)$, for $\tau \in (0,T^4)$, which vanishes in a neighborhood of $\partial_p\Omega_T$ being continuous and zero on $\partial_p\Omega_T$, we have
\[
- \int_{\hat B\times(0,\tau)}  \langle\widetilde{\mathcal{A}}(Dv), D \eta \rangle \dx \dt = \int_{\hat B\times(0,\tau)}  \partial_t \mathcal H(v) \eta \dx\dt.
\]
To be precise, this choice of the test function is admissible only after a suitable mollification in time; see for instance the steps in the end 
of the proof of \cite[Lemma 2.3]{BKU} for a rigorous treatment of the parabolic term in this setting. Indeed one should prove the estimate not directly up to $t=0$ but $t=\eps$, for $\varepsilon$ (the mollification parameter) small enough, and then pass to the limit. We have
\[
\partial_t \mathcal{H}(v) \varPsi^\prime (v) \varPsi(v) = \partial_t \int_{v_-}^v \mathcal{H}^\prime(\xi) \varPsi^\prime (\xi) \varPsi(\xi) \, d\xi
 \]
and integration by parts gives 
\[
\int_{\hat B\times(0,\tau)} \partial_t \mathcal{H}(v) \varPsi^\prime (v) \varPsi(v) \phi^p \dx \dt = \int_{\hat B} \int_{v_-}^{v(\cdot,t)}
\mathcal{H}^\prime(\xi) \varPsi^\prime (\xi) \varPsi(\xi) \, d\xi \phi^p \dx \bigg|_{t=0}^{\tau} \,,
\]
since $\phi$ is time independent and recalling that  $v \in C(\overline Q)$. Since $v < v_-$ on $\hat B \times \{0\}$, we have that the term on the right-hand side for $t=0$ is zero.
Therefore
\[
\int_{\hat B\times(0,\tau)} \partial_t \mathcal{H}(v) \varPsi^\prime (v) \varPsi(v) \phi^p \dx \dt  =  \int_{\hat B} \int_{v_-}^{v(x,\tau)}
\mathcal{H}^\prime(\xi) \varPsi^\prime (\xi) \varPsi(\xi) \, d\xi \phi(x)^p \dx
\]
and since $\mathcal H^\prime \geq 1$ and  $\varPsi(v_-)=0$, we obtain 
\[
\int_{\hat B} \varPsi^2 (v(\cdot,\tau))  \phi^p \dx \leq 2 \int_{\hat B\times(0,\tau)} \partial_t \mathcal{H}(v) \varPsi^\prime (v) \varPsi(v) \phi^p \,dx \dt.
\]
As for the elliptic term, we get from \eqref{tilde inf}
\begin{align*}
-\int_{\hat B\times(0,\tau)}\langle\widetilde{\mathcal{A}}(Dv), D\eta \rangle \dx \dt  &=- \int_{\hat B\times(0,\tau)}
\langle\widetilde{\mathcal{A}}(Dv),Dv\rangle (1+\varPsi(v)) \left[ \varPsi^\prime(v) \right]^2 \phi^p  \dx \dt \\
&\qquad- \int_{\hat B\times(0,\tau)}\langle\widetilde{\mathcal{A}}(Dv),D\phi^p\rangle  \, \varPsi^\prime(v) \varPsi(v)  \dx \dt\\
&\leq -c(p,\Lambda)\int_{\hat B\times(0,\tau)} |Dv|^p (1+\varPsi(v)) \left[ \varPsi^\prime(v) \right]^2 \phi^p  \dx \dt\\
&\qquad +c(p,\Lambda)  \int_{Q}  \varPsi(v) \left[ \varPsi^\prime(v) \right]^{2-p} |D\phi|^p \dx \dt,
\end{align*}
using Young's inequality. We thus obtain, discarding the negative term on the right-hand side
\[
\int_{\hat B} \varPsi^2 (v(\cdot,\tau))  \phi^p \dx \leq c  \int_{Q}   \varPsi(v) \left[ \varPsi^\prime(v) \right]^{2-p}|D\phi|^p \dx \dt;
\]
this holds for all $\tau \in (0,T^4]$.  The very definitions of $\varPsi$ and $T^4$ then imply
\[
\int_{\frac12 \hat B} \left[ \varPsi(v(\cdot,\tau)) \right]^2 \dx \leq   c  \, \frac{| \hat B |\,T^4}{r^p} \log \frac{1}{{\theta}} \Big(\frac\omega{8}\Big)^{p-2}\leq   c \,\big|{\textstyle \frac12} \hat B \big| \log \frac{1}{{\theta}},
\]
since ${\theta} \omega/8< (v-k)_+\leq\omega/8$ in $\{\varPsi(v) \neq 0\}$.
Moreover, the left-hand side can be bounded from below as
\[
\int_{\frac12\hat B} \left[ \varPsi(v(\cdot,\tau)) \right]^2 \dx \geq \left| {\textstyle \frac12}\hat B \cap \left\{ v(\cdot , \tau ) \geq \sup_Q v - \theta\, \omega/8 \right\} \right| \Big( \log \frac{1}{2{\theta}} \Big)^2
\]
and we conclude with
\[
\frac{\left| {\textstyle \frac12}\hat B \cap \left\{ v(\cdot , \tau ) \geq \sup_Q v - \theta \,\omega/8\right\} \right| }{| {\textstyle \frac12}\hat B  | } \leq c\, \frac{\log(1/\theta)}{[\log (1/(2\theta))]^2}\leq \frac{c}{\log (1/\theta)}\,.
\]
\end{proof}

Therefore, if \eqref{starrr} holds, then for all $\nu_\ast \in (0,1)$ we find $\eps_4 \equiv \eps_4(n,p,\Lambda,\nu_\ast)$ such that after integration, denoting $\sigma Q:=(B_{\sigma r}\cap\Omega)\times(0,T)$ for $\sigma\in(0,1]$, we have 
\begin{equation*}
\Big| {\textstyle \frac12} Q\cap \Big\{ v\geq  \sup_Qv - 2\eps_4\, \omega \Big\} \Big| \leq \nu_\ast \big| {\textstyle \frac12} Q\big| .
\end{equation*}
We can now deduce the following.

\begin{proposition}\label{initial.red.osc}
Let $v$ be a solution to \eqref{the equation for v.bnd} in $Q$ and suppose that \eqref{starrr} holds for some $\omega>0$. Then
\begin{equation}\label{initial.reduction}
\sup_{\frac14 Q} v \leq \sup_Q v  - \eps_4 \omega\,,
\end{equation}
where $\eps_4$ is a constant depending on $n,p,\Lambda,\delta$ and $q$.
\end{proposition}

\begin{proof}
Note that taking independent of time cut-off functions, the Caccioppoli's inequality does not contain the terms containing $H'_{b,\eps}$ on the right-hand side. In particular we set
\[
Q_j :=\Big(B_{\sigma_jr}(x_0)\cap\Omega\Big)\times (0,T^4)=:B_j \times (0,T^4),\qquad \sigma_j=\frac14\big(1+2^{-j}\big),
\]
and we have
\begin{multline*}
\frac1{T^4}\sup_{0<t<T^4}  \mean{B_{j+1} }  \big[(v-k)_+^2\big](\cdot,t) \dx   + \mean{Q_{j+1} } |D(v-k)_+|^p \dx\dt\\
\leq c \, 2^{jp}    \mean{Q_j } \frac{(v-k)_+^p}{r^p}  \dx \dt.
\end{multline*}
Setting $k_j:=\sup_{Q} v  - (1+2^{-j})\eps_4 \omega$ and using Sobolev's inequality \eqref{Sobolev} (possibly the boundary version mentioned in the last remark of Paragraph \ref{Sobo.paragraph}) we infer, with $\kappa$ defined in~\eqref{eq:kappa} and the agreement in \eqref{agree},
\begin{align*}
 \mean{Q_{j+1}} &(v-k_j)_+^{2(1-1/\kappa)+p} \dx \dt   \\
&\leq   c(n,p,\delta)\, r^p \big[T^4\big]^{1-1/\kappa} \biggl[ \frac{2^{jp}}{r^p}    \mean{Q_j} (v-k_j)^p_+ \dx \dt \biggr]^{2-1/\kappa} \notag\\
&\leq   c\, 2^{c(p,\kappa)j} \omega^{(1-1/\kappa)(2-p)+p(2-1/\kappa)}\biggl[  \mean{Q_j} \chi_{\{v > k_j\}} \dx \dt \biggr]^{2-1/\kappa}\notag.
\end{align*}
Note all this is possible since $k_j \geq \sup_{Q} v-\omega/8$ when $\eps_4$ is small enough, and hence $(v-k_j)_+$ vanishes in a neighborhood of $\partial_{\rm par} \Omega_T$ by the boundary continuity of $v$. Now reasoning as after \eqref{third.alt}, a standard hyper-geometric iteration lemma yields \eqref{initial.reduction} provided that $\nu_*$ is chosen small enough, depending on $n,p,\Lambda,\delta$ and $q$; this finally fixes $\eps_4$.
\end{proof}

\section{The approximate boundary continuity}
The goal of this Section is the iteration of the results of the previous Section; this will give in a standard way, as a consequence, the boundary continuity. Moreover, we shall show how to explicitly infer the modulus described in \eqref{choice.modulus.boundary}.

\subsection{Iterative estimates}
The goal of the next Proposition will be twofold. On the one hand, we show how to set the estimates \eqref{allinall} and \eqref{initial.reduction} into an iterative scheme. On the other hand, we unify the interior (presented in \cite{BKU}), initial and lateral boundary cases in order to have estimates slightly more manageable.

\begin{proposition}\label{interior.continuity.iterated}
Let $R_0\leq r_\Omega$, $(x_0,t_0)\in\overline{\Omega}_T$ and $q>\bar q$, where $\bar q\geq 2$ has been defined in \eqref{eq:lat q cond}; set
\[
\alpha:=\frac{1}{p'q}\in \Big(0,\frac{1}{p'\bar q}\Big).
\]
Then there exist constants $\vartheta,\tau\in(0,1/2)$ depending only on $n,p,\Lambda,\delta$ and $q$ such that for any decreasing sequence $\{\omega_j\}_{j\in\N_0}$ with 
\begin{equation}\label{omega_j}
\omega_0:=1,\qquad \omega_{j+1}\geq\omega_j\Big(1 - \vartheta \exp\big(-[\vartheta\omega_{j}]^{-1/\alpha}\big)\Big)
\end{equation}
and moreover defining for $j\in\N_0$ 
\begin{equation}\label{choices_j}
\begin{split}
&\wt\omega_j:=\tau\omega_j\exp\left(-[\tau\omega_j]^{-1/\alpha}\right)\\
&R_{j+1}:=\exp\left(-\frac{\vartheta}{\alpha}[\vartheta\omega_j]^{-1/\alpha}\right)R_j, \qquad T_{j} :=  \wt\omega_j^{1-p}R_j^p \\
&Q^j:=\left(B_{R_j}(x_0) \times\left(t_0-T_j,t_0 + T_j\right) \right) \cap \overline{\Omega}_T,
\end{split}
\end{equation}
we have the following: If $v$ is a continuous weak solution to \eqref{the equation for v.bnd} in $Q^j$  with $\eps\leq\wt\omega_j/2$ and such that
\begin{equation}\label{assumption_j}
\osc_{Q^j} v \leq \omega_j
\end{equation}
for some $j\in\N_0$, then 
\begin{equation}\label{claim}
\osc_{Q^{j+1}} v \leq \max\big\{\omega_{j+1},2\osc_{\overline Q^j\cap\partial_p\Omega_T}\tilde g\big\}.
\end{equation}
\end{proposition}
\begin{proof}
Fix $j\in\N_0$ as in the statement of the Proposition and suppose that \eqref{assumption_j} holds. Observe that by considering the time $t_j := t_0 + T_j$ instead of $t_0$, we may write both $Q^j$ and $Q^{j+1}$ as backwards in time cylinders: 
\begin{equation*} 
Q^i = \Big(B_{R_i}(x_0) \times (t_i - 2T_i,t_i)\Big)\cap \overline{\Omega}_T
\end{equation*}
for $i=j,j+1$. Notice that it could indeed happen that $t_j,t_{j+1}>T$. In order to have some freedom we choose two auxiliary parameters 
\begin{equation*} 
\wt R_j := \exp\left(-\frac{2\vartheta}{3\alpha}[\vartheta\omega_j]^{-1/\alpha}\right)R_j\qquad {\rm and}\qquad \hat R_j := \exp\left(-\frac{\vartheta}{3\alpha}[\vartheta\omega_j]^{-1/\alpha}\right)R_j.
\end{equation*}
Note that not only do we have $R_{j+1}\leq \wt R_j\leq \hat R_j\leq R_j$, but the ratios
\begin{equation}\label{ratii}
\frac{R_{j+1}}{\wt R_j}=\frac{\wt R_{j}}{\hat R_j}=\frac{\hat R_{j}}{R_j}=\exp\left(-\frac{\vartheta}{3\alpha}[\vartheta\omega_j]^{-1/\alpha}\right)\leq \exp\left(-\frac{\vartheta^{-1}}{3\alpha}\right) 
\end{equation}
can be made as small as we please by choosing $\vartheta$ small enough (note that $\alpha<1/4$). Moreover, we set 
\begin{equation*} 
Q_{\rm int}(r,\omega) := B_{r}(x_0) \times (\bar t- \tilde M \omega^{(2-p)(1+1/\tilde\alpha)} r^p ,\bar t),\quad \bar t:=\max\{t_{j+1},T\};
\end{equation*}
$\tilde\alpha\equiv \tilde\alpha(n,p)$ is the exponent appearing in \cite[Theorem 1.2]{BKU}, relabeled; its explicit value is not important here, only the fact that $\tilde\alpha\in(0,1)$. $\tilde M$ is the constant appearing in \cite[Theorem 1.2]{BKU}, larger than one and depending on $n,p,\Lambda$ and $\tilde\alpha$; note that the dependence on $\tilde\alpha$ is meaningful only in the case $p=n$. We fix, in this case, $\tilde\alpha=1/4$ so that in any case $\tilde M=\tilde M(n,p,\Lambda)$.

\vspace{3mm}

\noindent  \emph{Case 1. Interior estimate.}
Let us first assume that $Q_{\rm int}(\wt R_j,\omega_j)\subset \Omega_T$. Since $\wt R_j\leq R_j$ and 
\begin{align}\label{small.time.scales}
\frac{\tilde M\omega_j^{(2-p)(1+1/\tilde\alpha)}\wt R_j^p}{\wt\omega_j^{1-p}R_j^p}&=\tilde M\tau^{p-1}\omega_j^{1-(p-2)/\tilde\alpha}\exp\Big(-(p-1)[\tau\omega_j]^{-1/\alpha}\Big)\Big(\frac{\wt R_j}{R_j}\Big)^p\notag\\
&\leq \tilde M\tau^{(p-2)(1+1/\tilde\alpha)}\sup_{\varsigma\in(0,1)}\varsigma^{1-(p-2)/\tilde\alpha}\exp\Big(-(p-1)\varsigma^{-1/\alpha}\Big)\notag\\
&=:\tilde M \,\mathcal S(p,q)\tau^{(p-2)(1+1/\tilde\alpha)}\leq 1
\end{align}
for small enough $\tau\equiv \tau(n,p,\Lambda,q)$, we have $Q_{\rm int}(\wt R_j,\omega_j)\subset Q^j$. Using now~\cite[Remark 4.3]{BKU} and the proof of~\cite[Theorem 4.1]{BKU} we see that 
\begin{equation*} 
\osc_{Q_{\rm int}(\frac{1}{32} \wt R_j,\omega_{j})} v \leq \omega_{j+1} \,,
\end{equation*}
and the inclusion $Q^{j+1} \subset Q_{\rm int}(\frac{1}{32}  \wt R_j,\omega_{j})$ follows choosing small enough $\vartheta$ depending on $n,p,\Lambda,q$ and $\tau$. Indeed, first we take $\vartheta$ so that $e^{-\vartheta^{-1}/[3\alpha]}\leq 1/32$ (see \eqref{ratii}). Then we notice that, since $\omega_{j+1}\geq\omega_j/2$,
\begin{align*}
\frac{\wt\omega_{j+1}^{1-p}R_{j+1}^p}{\tilde M\omega_j^{(2-p)(1+1/\tilde\alpha)}(\wt R_j/32)^p}&\leq c(p,\tilde M)\Big(\frac{\tau\omega_j}2\Big)^{1-p}\frac{\exp\big((p-1)\big(\frac{\tau\omega_j}2\big)^{-1/\alpha}\big)}{\omega_j^{(2-p)(1+1/\tilde\alpha)}}\Big(\frac{R_{j+1}}{\wt R_j}\Big)^p\\
&\leq \frac{c}{\tau^{p-1}\omega_j}\exp\bigg(\Big[-\frac{p\vartheta^{1-1/\alpha}}{3\alpha}+(p-1)\big(\frac{\tau}2\big)^{-1/\alpha}\Big]\omega_j^{-1/\alpha}\bigg).
\end{align*}
Now if we choose $\vartheta\leq c(p,\alpha)\tau^{1/(1-\alpha)}$ yielding
\[
\frac{\wt\omega_{j+1}^{1-p}R_{j+1}^p}{\tilde M\omega_j^{(2-p)(1+1/\tilde\alpha)}(\wt R_j/32)^p}\leq\frac{c(n,p,\Lambda,q)}{\vartheta^{p-1}\omega_j}\exp\bigg(\Big[-\frac{p\vartheta^{1-1/\alpha}}{6\alpha}\Big]\omega_j^{-1/\alpha}\bigg),
\]
then this quantity can be made smaller than one by choosing $\vartheta$ further small. Note that when we decrease the value of $\tau$ in what follows, we shall decrease also the value of $\vartheta$ accordingly. 

\smallskip \noindent \emph{Case 2. Initial boundary.}
Suppose that $\overline{Q_{\rm int}(\wt R_j,\omega_j)}$ touches the initial boundary, that is, $t_{j+1}\leq \tilde M \omega_j^{(2-p)(1+1/\tilde\alpha)} \wt R_j^p$. We define
\[
Q_{\rm ini}(\hat R_j,\omega_j):=\big(B_{\hat R_j}(x_0)\cap\Omega\big)\times(0,\omega_j^{2-p}\hat R_j^p)
\]
and we assume that
\begin{equation}\label{bound.nontrivial.1}
\sup_{\overline Q^j\cap\partial_p\Omega_T}\tilde g\leq \sup_{Q_{\rm ini}(\hat R_j,\omega_j)}v-\frac{\omega_j}{8}\quad \Longrightarrow\quad\sup_{B_{\hat R_j}\cap \Omega}v(\cdot,0)\leq \sup_{Q_{\rm ini}(\hat R_j,\omega_j)}  v-\frac{\omega_j}{8} 
\end{equation}
holds. We are thus in a position to apply Lemma~\ref{lemma: log lemma} and to subsequently infer~\eqref{initial.reduction}:
\[
\sup_{(B_{\hat R_j/4}\cap\Omega)\times(0,\omega_j^{2-p}\hat R_j^p)} v \leq \sup_{Q_{\rm ini}(\hat R_j,\omega_j)} v  - \eps_4 \omega_j\leq \sup_{Q^j} v  - \eps_4 \omega_j;
\]
the last inequality holds, since $\hat R_j\leq R_j$ and $\omega_j^{2-p}\leq \wt \omega_j^{1-p}$. Since $R_{j+1}\leq \frac14 \hat R_j$ and 
\[
t_{j+1}\leq \tilde M\omega_j^{(2-p)(1+1/\tilde\alpha)}\wt R_j^p \leq \omega_j^{2-p}\hat R_j^p
\]
for $\vartheta$ small, we also have $Q^{j+1}\subset (B_{\hat R_j/4}(x_0)\cap\Omega)\times(0,\omega_j^{2-p}\hat R_j^p)$. Note indeed that
\begin{align*}
\frac{\tilde M\omega_j^{(2-p)(1+1/\tilde\alpha)}}{\omega_j^{2-p}}\Big(\frac{\wt R_j}{\hat R_j}\Big)^p&\leq \tilde M\omega_j^{(2-p)/\tilde\alpha}\exp\left(-\frac{\vartheta}{3\alpha}[\vartheta\omega_j]^{-1/\alpha}\right)\\
&\leq \tilde M\vartheta^{(p-2)/\tilde\alpha}\sup_{\varsigma>0}\varsigma^{(2-p)/\tilde\alpha}\exp\Big(-\frac{\varsigma^{(\alpha-1)/\alpha}}{3\alpha}\Big)\leq 1
\end{align*}
for small enough $\vartheta$. Thus,
\begin{equation}\label{red.osc.ini}
\sup_{Q^{j+1}} v\leq \sup_{Q^j} v  - \eps_4 \omega_j\quad \Longrightarrow\quad\osc_{Q^{j+1}} v\leq \osc_{Q^j} v  - \eps_4 \omega_j\leq \omega_j(1-\eps_4)\leq \omega_{j+1}, 
\end{equation}
after having subtracted from both sides $\inf_{Q^{j+1}} v$ and taking $\vartheta\leq \eps_4$. The case 
\begin{equation}\label{bound.nontrivial.2}
\inf_{\overline Q^j\cap\partial_p\Omega_T}  \tilde g\geq \inf_{Q_{\rm ini}(\hat R_j,\omega_j)}  v+\frac{\omega_j}{8} 
\end{equation}
can be reduced to the previous one simply observing that $-v$ satisfies an equation structurally similar to \eqref{the equation for v.bnd} with $-\tilde g$ replacing $\tilde g$ as boundary datum; thus also in this case we conclude with \eqref{red.osc.ini}. To conclude, note that we may assume that $\osc_{Q^{j+1}}v> \frac34 \omega_j$, because otherwise 
\[
\osc_{Q^{j+1}}v\leq {\frac34} \omega_j\leq\omega_{j+1}.
\]
Thus, if neither \eqref{bound.nontrivial.1} nor \eqref{bound.nontrivial.2} holds, subtracting the converse inequalities gives
\[
\osc_{\overline Q^j\cap\partial_p\Omega_T}\tilde g\geq\osc_{Q_{\rm ini}(\hat R_j,\omega_j)}v-\frac{\omega_j}{4}\geq\osc_{Q^{j+1}}v-\frac{\omega_j}{4}\geq\frac{\omega_j}{2};
\]
in view of \eqref{assumption_j} this implies $\osc_{Q^{j+1}}v\leq\osc_{Q^j}v\leq 2\osc_{\overline Q^j\cap\partial_p\Omega_T}\tilde g$.

\smallskip \noindent \emph{Case 3. Lateral boundary.} We finally assume that $B_{\wt R_j}(x_0) \cap \partial \Omega \neq \emptyset$. The idea is to use the results of Section~\ref{red.lateral} with $\omega=\omega_j$, $r=\hat R_j$ and $\eps_1=\tau$, which yield 
\[
\wt\omega=\wt\omega_j,\qquad T^1=(\tau\omega_j)^{2-p}\hat R_j^p, \qquad T^3=\wt\omega_j^{1-p}\hat R_j^p.
\]
Since $x_0$ is close to the boundary, we find $\wt x\in\partial\Omega$ such that $|x_0-\wt x|\leq \wt R_j$, and thus for a small $\vartheta$ we have $B_{R_{j+1}}(x_0)\subset \frac1{16}B_{\hat R_j}(\wt x)$ using \eqref{ratii}. Moreover, we estimate 
\[
\begin{split}
\frac{2T_{j+1}}{\frac1{16}T^1}&=32\,\Big(\frac{\omega_j}{\omega_{j+1}}\Big)^{p-2}\frac1{\tau\omega_{j+1}}\exp\left((p-1)[\tau\omega_{j+1}]^{-1/\alpha}\right)\bigg(\frac{R_{j+1}}{\hat R_j}\bigg)^p\\
&\leq 2^{p+3}\exp\bigg(\bigg(\Big(\frac{2}{\tau}\Big)^{1/\alpha}-\frac{2}{3\alpha}\vartheta^{(\alpha-1)/\alpha}\bigg)p\,\omega_j^{-1/\alpha}\bigg)\leq 1
\end{split}
\]
for small enough $\vartheta$ using $\omega_j\leq 2\omega_{j+1}$ (implied by \eqref{omega_j}) and $1/x\leq\exp(x^{-1/\alpha})$ for $x>0$; recall that $\alpha\in (0,1)$. Therefore $Q^{j+1}\subset\frac{1}{16}Q^1(\wt x)$, where 
\[
Q^1(\wt x) := \left(B_{\hat R_j}(\wt x)\times\left(t_{j+1}-T^1,t_{j+1}\right)\right)\cap{\Omega}_T;
\]
moreover, we clearly have $Q^3(\wt x)\subset Q^j$ for small $\vartheta$, if we set 
\[
Q^3(\wt x):=\left(B_{\hat R_j}(\wt x) \times\left(t_{j+1}-T^3,t_{j+1}\right) \right) \cap{\Omega}_T.
\]
Now we assume that
\[
\sup_{\overline{Q^3}(\wt x)\cap\partial_p\Omega_T}\tilde g\leq \sup_{Q^3(\wt x)}v-\frac{\omega_j}{8}.
\]
Possibly reducing the value of $\vartheta$ and noting that the map $\sigma\mapsto \exp\left(-\sigma^{-1/\alpha}\right)$ is increasing, \eqref{allinall} gives 
\begin{equation}\label{oscreduction}
\osc_{\frac1{16}Q^1(\wt x)} v \leq \osc_{Q^3(\wt x)} v -\vartheta\omega_j\exp\left(-[\vartheta\omega_j]^{-1/\alpha}\right)\,;
\end{equation}
note that we are assuming $\eps\leq\wt\omega_j/2$. Using \eqref{assumption_j} and \eqref{omega_j}, we can bound the right-hand side of \eqref{oscreduction} by $\omega_{j+1}$, which gives the result. The case
\[
\inf_{\overline{Q^3}(\wt x)\cap\partial_p\Omega_T}\tilde g\geq \inf_{Q^3(\wt x)}v+\frac{\omega_j}{8}
\]
is handled similarly; see Remark \ref{vartheta.fixed.2}. In the remaining case we have, similarly to {\em Case 2}, that either 
\[
\osc_{Q^{j+1}}v\leq {\frac34} \omega_j\leq\omega_{j+1}
\]
or 
\[
\osc_{Q^{j+1}}v\leq \osc_{Q^j}v\leq\omega_j\leq 2\osc_{\overline{Q^3}(\wt x)\cap\partial_p\Omega_T}\tilde g.
\]
This concludes the proof of \eqref{claim}; $\vartheta$ and $\tau$ are now fixed as constants depending only on $n,p,\Lambda,\delta$ and $q$.
\end{proof}

\begin{claim}\label{claim.Tuomo}
Once fixed $R_0>0$ and $\alpha$ as in \eqref{beta}, with $\omega(\cdot)$ defined in \eqref{choice.modulus.boundary}, $\vartheta$ fixed in Proposition \ref{interior.continuity.iterated}, $\lambda_0:=\exp(\exp(\vartheta^{-1/\alpha}))$ and $R_j$ defined in \eqref{choices_j}, the sequence $\{\omega_j\}_{j\in\N_0}$ with the choice $\omega_j:=\omega(R_j)$ satisfies \eqref{omega_j}; that is
\[
\omega(R_0)=1,\qquad \omega(R_{j+1})\geq\omega(R_j)\Big(1-\vartheta \exp\big(-[\vartheta\omega(R_j)]^{-1/\alpha}\big)\Big).
\]
Moreover,
\begin{equation}\label{doubling}
\omega(R_j)\leq 2\omega(R_{j+1}). 
\end{equation}
\end{claim}
\begin{proof}
First, we obviously have $\omega(R_0)=1$ by the choice of $\lambda_0$. For any fixed $j\in \N_0$, using the elementary inequality $1-x\leq e^{-x}$ that is valid for any $x$, we have
\[
1-\vartheta \exp\big(-[\vartheta\omega(R_j)]^{-1/\alpha}\big) \leq \exp\Big(-\vartheta \exp\big(-[\vartheta\omega(R_j)]^{-1/\alpha}\big)\Big).
\]
Now we estimate the argument of the exponential in the following way:
\[
-\vartheta \exp\big(-[\vartheta\omega(R_j)]^{-1/\alpha}\big)\leq -\frac{\vartheta}{\int_{R_{j+1}}^{R_j}\frac{d\rho}{\rho}}\int_{R_{j+1}}^{R_j}\exp\big(-[\vartheta\omega(\rho)]^{-1/\alpha}\big)\frac{d\rho}{\rho}
\]
since the map $\rho\mapsto - \exp(-[\vartheta\omega(\rho)]^{-1/\alpha})$ is decreasing. We compute, using the expression in \eqref{choices_j} for $R_{j+1}/R_j$,
\[
\int_{R_{j+1}}^{R_j}\frac{d\rho}{\rho}=-\log\Big(\frac{R_{j+1}}{R_j}\Big)=\frac{\vartheta}{\alpha}[\vartheta\omega(R_j)]^{-1/\alpha}
\]
and using \eqref{choice.modulus.boundary} for the explicit expression of $\omega(\cdot)$
\[
\exp\big(-[\vartheta\omega(\rho)]^{-1/\alpha}\big)=\exp\bigg(-\log\Big(\log\Big(\frac{\lambda_0 R_0}\rho\Big)\Big)\bigg)=\frac{1}{\log(\lambda_0 R_0/\rho)}.
\]
Thus, merging the estimates above and using again the aforementioned monotonicity and the expression for $\omega(\cdot)$, we have
\begin{align*}
1-\vartheta \exp\big(\!\!-[\vartheta\omega(R_j)]^{-1/\alpha}\big)&\leq \exp\bigg(-\alpha[\vartheta\omega(R_j)]^{1/\alpha}\int_{R_{j+1}}^{R_j}\frac{1}{\log(\lambda_0 R_0/\rho)}\frac{d\rho}{\rho}\bigg)\\
&\leq\exp\bigg(\!\!-\alpha\int_{R_{j+1}}^{R_j}[\vartheta\omega(\rho)]^{1/\alpha}\frac{1}{\log(\lambda_0 R_0/\rho)}\frac{d\rho}{\rho}\bigg)\\
&=\exp\bigg(\!\!-\alpha\int_{R_{j+1}}^{R_j}\frac{1}{\log(\log(\lambda_0 R_0/\rho))}\frac{1}{\log(\lambda_0 R_0/\rho)}\frac{d\rho}{\rho}\bigg).
\end{align*}
We conclude by simply computing
\[
-\alpha\int_{R_{j+1}}^{R_j}\frac{1}{\log(\log(\lambda_0 R_0/\rho))}\frac{1}{\log(\lambda_0 R_0/\rho)}\frac{d\rho}{\rho}=\log\left(\bigg[\frac{\log(\log(\lambda_0 R_0/R_{j}))}{\log(\log(\lambda_0 R_0/R_{j+1}))}\bigg]^\alpha\right)
\]
and since the last quantity is equal to $\log(\omega(R_{j+1})/\omega(R_j))$, the first part of the Claim is proved. For the doubling property \eqref{doubling}, it is enough to recall that $\vartheta, \omega(R_j)\leq 1$.
\end{proof}

\subsection{Uniform modulus of continuity}
We finally prove that our approximate solution $v$ is almost equi-continuous. In order to fix a normalization condition, we assume that
\begin{equation}\label{normalization}
\osc_{\Omega_T} v\leq 1
\end{equation}
holds true.
\begin{proposition} \label{p.equicont}
Suppose that $v$ is a weak solution to \eqref{the equation for v.bnd} in $\Omega_T$ attaining continuously the boundary values $\tilde g \in C^0(\partial_p \Omega_T)$ on $\partial_p \Omega_T$, and suppose that \eqref{normalization} holds true. There is a modulus of continuity $\bar \omega$ independent of $\eps$ and $(x_0,t_0) \in \overline{\Omega}_T$ such that 
\begin{equation} \label{e.explicit.u_eps}
\osc_{Q_r(x_0,t_0)  \cap \overline{\Omega}_T} v \leq \bar \omega(r) + h(\eps)
\end{equation}
for every $r>0$, where $h(\eps) \to 0$ as $\eps \to 0$. 
\end{proposition}

\begin{proof}
Fix $(x_0,t_0) \in \overline{\Omega}_T$. We first examine the case $r \leq r_\Omega^2$ and look for a modulus of continuity $\bar \omega(\cdot)$ of the form
\begin{equation}\label{e.explicit.u_eps2}
\bar \omega(r) := 4c_0\big(\omega_1(\sqrt r) + \omega_{\tilde g}(\sqrt r) \big)\,,
\end{equation}
where $c_0>1$ is to be chosen, $\omega_1(\cdot)$ is $\omega(\cdot)$ defined in~\eqref{choice.modulus.boundary} with $R_0=1$, and $\omega_{\tilde g}(\cdot)$ is a modulus of continuity for $\tilde g$, as described in \eqref{mod.g}.  The error $h(\cdot)$ is defined as the unique solution to
\begin{equation*} 
\tau h(\eps) \exp\left(-[\tau h(\eps)]^{-1/\alpha}\right) = 2\eps\,,
\end{equation*}
and it is easy to see that $h(\eps) \to 0$ as $\eps \to 0$. 

\vs

According to Proposition~\ref{interior.continuity.iterated}, we define inductively
\begin{align}\label{def.iter}
&R_0 := \sqrt r,\qquad \omega_{j}:= \omega(R_j),\qquad  \wt\omega_j:=\tau\omega_j\exp\left(-[\tau\omega_j]^{-1/\alpha}\right),\notag\\
&R_{j+1}:=\exp\left(-\frac{\vartheta}{\alpha}[\vartheta\omega_j]^{-1/\alpha}\right)R_j,\\
&Q^j:=\left(B_{R_j}(x_0)\times(t_0-\wt\omega_j^{1-p}R_j^p,t_0+\wt\omega_j^{1-p}R_j^p)\right)\cap\overline\Omega_T.\notag
\end{align}
Here $\omega(\cdot)$ is defined as in \eqref{choice.modulus.boundary} with the expression of $R_0$ above. We then fix $\bar \jmath$ as the largest index $j$ for which $\eps\leq \wt\omega_j/2$; then $\omega_{j}<h(\eps)$ for all $j\geq \bar \jmath+1$.  Moreover we let $k \in \N_0$ be such that $R_{k+1} \leq r < R_k$ and denote $Q:=Q_r(x_0,t_0)\cap\overline\Omega_T$.
Finally, note that if $\osc_{Q} v \leq 4 \omega_{\tilde g}(c_0\sqrt r)$, then there is nothing to prove, by the concavity of $\omega_{\tilde g}$. Thus we may assume that $\osc_{Q} v > 4  \omega_{\tilde g}(c_0\sqrt r)$ in the rest of the proof. 

\vspace{3mm}

 We proceed inductively, showing that 
\begin{equation}\label{e.ind.k}
\osc_{Q^j} v \leq \omega_j
\end{equation}
for all $j\in \{0,\dots, \min\{\bar\jmath+1,k+1\}\}$. Note that \eqref{e.ind.k} is certainly true for $j=0$ in view of  \eqref{normalization}, since $\omega_0 = 1$. Assume that \eqref{e.ind.k} holds for some $j\in \{0,\dots, \min\{\bar\jmath,k\}\}$. Since
 \begin{equation*} 
\omega_j \geq \osc_{Q^j} v > \osc_{Q} v > 4 \omega_{\tilde g}(c_0\sqrt r)
\end{equation*}
by the fact that $R_j\geq R_k\geq r$, we see, by the doubling property $ \omega_{j} \leq 2\omega_{j+1}$ given by Claim~\ref{claim.Tuomo}, that
\begin{equation}\label{e.cond.on.tildeg}
2 \osc_{\overline Q^j \cap \partial_p \Omega_T} \tilde g \leq 2\omega_{\tilde g}(c_0\sqrt r) \leq \frac12 \omega_{j} \leq \omega_{j+1}\,.
\end{equation}
The first inequality follows by choosing $c_0=\wt\omega_0^{(1-p)/p}$, $\wt\omega_0\equiv \wt\omega_0(n,p,\Lambda,q)$, since clearly 
\[
\wt\omega_j^{1-p}R_j^p\leq \wt\omega_0^{1-p}R_0^p=\wt\omega_0^{1-p}{(\sqrt r)}^p={(c_0\sqrt r)}^p.
\]
Thus, since $j\leq \bar\jmath$ and therefore $\wt\omega_j\geq 2\eps$, by Proposition~\ref{interior.continuity.iterated} and Claim~\ref{claim.Tuomo} (note that $R_0\leq r_\Omega$), we obtain
\begin{equation*} 
\osc_{Q^{j+1}} v \leq \max\big\{\omega_{j+1},2 \osc_{\overline Q^j \cap \partial_p \Omega_T} \tilde g\big\}=\omega_{j+1},
\end{equation*}
proving the induction step. 

\vs

Now, if we have $\bar\jmath\geq k$, then \eqref{e.ind.k} holds in particular for $j=k$. If on the other hand $\bar\jmath< k$ we use \eqref{e.ind.k} with $j=\bar \jmath+1$; in this case,  we have 
\[
\osc_{Q^k} v \leq \osc_{Q^{\bar\jmath+1}} v\leq\omega_{\bar\jmath+1}<h(\eps).
\]
As a consequence, merging the two cases, we get
\begin{equation}\label{e.iiiii}
\osc_{Q_r(x_0,t_0)\cap\Omega_T} v \leq \osc_{Q^k} v \leq \omega_k + h(\eps) \leq 2 \omega(R_{k+1})+ h(\eps) \leq 2\omega(r)+ h(\eps)
\end{equation}
and this essentially finishes the proof, since by the definition of $R_0$ in $\omega(\cdot)$, $\omega(r)=\omega_1(\sqrt r)$. On the other hand, if $r >r_\Omega^2$, then by \eqref{normalization}
\[
\osc_{Q_r(x_0,t_0)  \cap \overline{\Omega}_T} v \leq \frac{r}{r_\Omega^2}\leq c\,\omega_1(\sqrt r).
\]
\end{proof}

We can also give a quantified version of the previous result. Set, for $r>0$
\begin{align*}
&\wt\omega(r):=\tau \omega(r)\exp\left(-[\tau \omega(r)]^{-1/\alpha}\right), \\
&\wt Q_r^{\omega(\cdot)}(x_0,t_0):=\left( B_{r}(x_0) \times(t_0-\wt\omega(r)^{1-p} r^p,t_0+\wt\omega(r)^{1-p} r^p) \right) \cap \overline{\Omega}_T.
\end{align*}

\begin{proposition} \label{p.equicont2}
Suppose that $v$ is a weak solution $v$ to \eqref{the equation for v.bnd} as in Proposition \ref{p.equicont}, $R_0>0$ and moreover that the boundary value function $\tilde g$ has the ``intrinsic'' modulus of continuity
\begin{equation} \label{e.g_quantitative}
\osc_{\wt Q_{r}^{\omega(\cdot)}(x_0,t_0)} \tilde g \leq \tilde M \omega(r)\,,
\end{equation}
for any $r\leq R_0$, with some $\tilde M>0$, where $(x_0,t_0) \in \overline{\Omega}_T$ and $\omega(\cdot)$ has been defined in~\eqref{choice.modulus.boundary}.
Then
\begin{equation}\label{e.llll}
\osc_{\widetilde Q_r^{\omega(\cdot)}(x_0,t_0)} v \leq c(p,\tilde M)\Big[ \omega(r) + h\Big(\frac\eps{4\tilde M+1}\Big) \Big]\,,
\end{equation}
for any $r>0$ with $h(\eps) \to 0$ as $\eps \to 0$. 
\end{proposition}
\begin{proof}
We can rescale $v$, solution to \eqref{the equation for v.bnd}, as we rescaled the solution to \eqref{Approx.CD.reg} in Paragraph \ref{rescaling}, with $\lambda: = 4\tilde M+1$. Then \eqref{e.g_quantitative} implies
\[
\osc_{\wt Q_r^{\omega(\cdot)}(x_0,t_0)} \hat g \leq \frac14 \omega(r)\,.
\]
Going back to \eqref{e.cond.on.tildeg} in the proof of Proposition~\ref{p.equicont} we see that  $2\osc_{Q_j} \hat g \leq \omega_{j+1}$ is always true due to the condition~\eqref{e.g_quantitative} and the definitions in \eqref{def.iter} (in particular, $\omega_{j}:= \omega(R_j)$). We are still in position to apply Proposition~\ref{p.equicont} to the solution $\hat v$ of \eqref{Approx.CD.reg} since $\int_\R \tilde H'(\sigma)\,d\sigma=\lambda^{-1}\leq1$; indeed, all the proofs of Section \ref{sec.three} are based only on the properties in \eqref{support.H} of $H_{b,\eps}$. Thus we obtain \eqref{e.ind.k} for any $j\in\N_0$ and then \eqref{e.iiiii} for $\hat v$ using an argument  analogous to the one in the proof of Proposition~\ref{p.equicont}.  Scaling back to $v$ gives \eqref{e.llll}.
\end{proof}

\section{The convergence proof}
In this section we conclude the proof of our main theorems. We first show that our approximants converge to a continuous function which is a physical solution of the problem thus proving Theorem \ref{main}. Then we see that the solution we built has the modulus of continuity \eqref{choice.modulus.boundary}, which gives Theorem \ref{boundary.continuity}.
\subsection{The Ascoli-Arzel\`a-type argument}\label{AAArgument}
We recall that $u_\eps$ solves the regularized Cauchy-Dirichlet probem
\[
\begin{cases}
\partial_t \big[\beta(u_\eps)+H_{a,\eps}(\beta (u_\eps))\big] - \mathrm{div}\, \mathcal{A}(x,t,u_\eps,Du_\eps) = 0\,,\qquad&\text{in $\Omega_T$},\\[5pt] 
u_\eps= g&\text{on $\partial_p\Omega_T$},
\end{cases}
\]
where the regularization of $H_a$ has been defined in \eqref{regularized}. 

\vspace{3mm} 

Note that by the maximum principle we have 
\begin{equation}\label{unibound}
\sup_{\Omega_T}|u_\eps|\leq\sup_{\partial_p\Omega_T} |g|
\end{equation}
independently of $\eps$; moreover $u_\eps$ is continuous up to the boundary and it has an ``equi-almost-uniform'' modulus of continuity in the following sense:  there exists a modulus of continuity $\bar\omega:[0,\infty)\to[0,\infty)$, concave and continuous, such that $\bar \omega(0)=0$ and for every $z,z'\in \overline{\Omega}_T$ and $\eps\in(0,1]$ it holds
\begin{equation}\label{equi.modulus}
|u_\eps(z)-u_\eps(z')|\leq\bar\omega(|z-z'|)+ h(\eps),
\end{equation}
the function $h(\cdot)$ having the property that it vanishes as $\eps\to0$.

To prove \eqref{equi.modulus}, first we define $w_\eps=\beta(u_\eps)$ as in Paragraph \ref{sec: approx}. Then we further rescale as in Paragraph \ref{rescaling} with $(x_0,t_0)=(0,0)$ and $\lambda:=\max\{\beta(2\|g\|_{L^\infty}),1\}$; call the rescaled function $v$ instead of $\hat v$. In fact $v$ solves \eqref{the equation for v.bnd} with $b=a/\lambda$ and $\Omega_T$ replaced by $\Omega\times(0,\lambda^{p-2}T)$, and the normalization condition \eqref{normalization} is clearly satisfied. We can thus make use of Proposition \ref{p.equicont} say, and this in turn yields 
\[
|v(z)-v(z')|\leq\bar\omega(|z-z'|)+ h(\eps);
\]
for some modulus of continuity $\bar\omega(\cdot)$; we can take for $|z-z'|$ the Euclidean distance in $\R^{n+1}$ without loss of generality, by suitably modifying $\bar\omega(\cdot)$.  Note that we have to use \eqref{unibound} too. This, in view of the Lipschitz regularity of $\beta$ yields \eqref{equi.modulus} where we avoided relabeling the quantities on the right-hand side. 

Call now $u_i, i\in\N$, the function $u_{\eps_i}$ for the choice $h(\eps_i)  \leq 1/i$; the sequence $\{u_i\}$ is equibounded thanks to \eqref{unibound} and, taking into account \eqref{equi.modulus}, satisfies
\begin{equation}\label{cont.partic}
|u_i(z)-u_i(z')|\leq \bar \omega\big(|z-z'|\big)+1/i
\end{equation}
for any $z,z'\in \overline{\Omega}_T$. If we consider the numerable dense subset $\mathcal S:=\overline{\Omega}_T\cap\mathbb Q^{n+1}$, by a standard diagonal argument, as a consequence of \eqref{unibound}, we extract a subsequence, still denoted by $\{u_i\}_{i\in\mathbb N}$, converging pointwise in $\mathcal S$ to $u$.  Moreover, slightly modifying the proof of Ascoli-Arzel\`a (see for instance the proof given in \cite[Page 17]{DiBePDE}), using condition \eqref{cont.partic} instead of equi-continuity, we show that the sequence $\{u_i\}_{i\in \N}$ actually converges pointwise in $\overline{\Omega}_T$ to a function which we shall call $u$, and moreover, by a similar argument, the convergence is uniform. In particular, $uÊ\in C^0(\overline{\Omega}_T)$.  The rest of the proof will be devoted in proving that $u$ is a local weak solution of \eqref{the general equation}. Assuming this for a moment, we now prove Theorem \ref{boundary.continuity}.

\subsection*{Proof of Theorem \ref{boundary.continuity}}
The constant $\vartheta$ is the one in Proposition \ref{interior.continuity.iterated} while $\lambda_0$ has been fixed in Lemma \ref{claim.Tuomo}. The constant $\tilde\delta\in(0,1)$ is defined, according to Proposition \ref{interior.continuity.iterated}, as 
\[
\tilde\delta^{2-p}:=\wt\omega_0^{1-p}=\tau^{1-p}\exp\left((p-1)\tau^{-1/\alpha}\right).
\]
To prove \eqref{modulus.u} we distinguish two cases. If $\omega_0\leq 1$, then \eqref{ass.ug}$_1$ directly implies that $\osc_{\tilde Q} u\leq 1$ with $\tilde Q:=(B_{R_0}(x_0)\times (t_0-\tilde\delta^{2-p}R_0^p, t_0+\tilde\delta^{2-p}R_0^p))\cap\Omega_T$ and we recognize that, by our choice of $\tilde\delta$, the cylinder $\tilde Q$ is exactly the cylinder $Q^0$ appearing in Proposition \ref{interior.continuity.iterated}. In the proof of Proposition \ref{p.equicont} we can clearly replace the renormalization in \eqref{normalization} with this local information, which is sufficient to start the iteration. On the other hand, jumping to Proposition \ref{p.equicont2},  we note that for $\gamma\in(0,1)$ as in \eqref{ass.ug}$_2$, $\wt\omega(r)^{1-p} \leq c_\gamma (r/R_0)^{-p\gamma}$, $c_\gamma$ depending on data, $\gamma$ and $R_0$, so
\begin{align*}
 \osc_{\wt Q_{r}^{\omega(\cdot)}(x_0,t_0)}\tilde g &\leq\osc_{B_r(x_0)\times(t_0-c_\gamma (r/R_0)^{p(1-\gamma)}, t_0+c_\gamma (r/R_0)^{p(1-\gamma)})} \tilde g\\
& \leq c(c_\gamma,R_0)\tilde M \omega_g\big((r/R_0)^{1-\gamma}\big)\leq c\,\omega(r)\,,
\end{align*}
and \eqref{e.g_quantitative} is satisfied. Thus we have \eqref{e.llll} at hand for $v=\beta(w_{1/i})$ and, after passing to the limit as $\eps=1/i\searrow 0$, we infer \eqref{modulus.u}.

In the case $\omega_0> 1$, the proof is exactly the same except for the fact that before starting we rescale $u,g$ to $\hat u,\hat g$ as in Paragraph \ref{rescaling} with $\lambda=\omega_0$. We again obtain $\osc_{Q^0}v\leq 1$ and \eqref{e.g_quantitative}, and we conclude by invoking \eqref{e.llll} and scaling back to $u$. Note that $Q_r^{\omega_0}\supset Q_{\omega_0^{(2-p)/p}r}$ and \eqref{modulus.u} holds trivially if $r>\omega_0^{(2-p)/p}R_0$.
\vs
 
In the following paragraph we show that the pointwise limit $u$ is a physical solution to our problem, that is, it satisfies the weak formulation of Definition \ref{deff}. 
\subsection{Convergence away from the jump}
We consider the previously defined sequence $\{u_i\}_{i\in\N}$ which converges uniformly in $\Omega_T$ to $u$. Here it is more convenient to work with $w_i:=\beta(u_i)$, which solves
\begin{equation}\label{final.equation}
\partial_t w_i - \mathrm{div} \, \bar{\mathcal{A}}(x,t,w_i,Dw_i)  = -\partial_t H_{a,\eps_i}(w_i)
\end{equation}
locally in $\Omega_T$, with $\bar{\mathcal A}$ having the same structure as $\mathcal A$, see \eqref{eq:mollified vf}. Note that also $\{w_i\}$ converges uniformly to $w=\beta(u)$. The reader might recall now that
\[
{\rm supp}\, H_{a,\eps_i}'(\cdot)\subset(a-\eps_i,a+\eps_i);
\]
hence, in the set $\Omega_T\cap \{|w_i-a|\geq \eps_i\}$ $w_i$ is a solution to a $p$-Laplacian-type equation 
\begin{equation}\label{plap}
\partial_t w_i -  \mathrm{div}\, \bar{\mathcal{A}}(x,t,w_i,Dw_i) =0. 
\end{equation}
Now we fix $\sigma>0$. By the uniform convergence, there exists $\bar n=\bar n(\sigma,\omega)$ such that
\[
\Omega_T\cap \{|w-a|\geq 2\sigma\}\subset \Omega_T\cap \{|w_i-a|\geq \sigma\}\subset \Omega_T\cap \{|w_i-a|\geq \eps_i\}
\]
for all $i \geq\bar n$; indeed 
\[
|w_i(x)-a|\geq |w-a|-|w_i(x)-w(x)|\geq2\sigma-\sigma=\sigma
\]
if (and independently of) $x\in\Omega_T\cap \{|w-a|\geq 2\sigma\}$ and if $i\geq\bar n$ is large enough, by the uniform convergence of $w_i$ to $w$. Hence $w_i$ is a solution to an evolutionary $p$-Laplacian-type equation in $\Omega_T\cap \{|w-a|\geq 2\sigma\}$ for all $i\geq\bar n$. Therefore using an argument similar to the proof of \cite[Theorem 5.3]{Kuusi10}, we find not only that $\{w_i\}_{i\geq\bar n}$ converges to $w$ uniformly in $\Omega_T\cap \{|w-a|\geq 2\sigma\}$, but also $Dw_i\to Dw$ almost everywhere in this set (and moreover $Dw\in L^p_{\rm loc}(\Omega_T\cap \{|w-a|\geq 2\sigma\})$); we shall sketch the proof in the next Paragraph. At this point, using a diagonal argument, we get that there exists a subsequence of the $\{w_i\}$ defined in Paragraph \ref{AAArgument}, still denoted by $\{w_i\}$, such that $w_i$ converges uniformly to $w$ in $\Omega_T$ and moreover
\[
Dw_i\to Dw\qquad\text{almost everywhere in ${|w-a|>0}$.}
\]
    
\subsection{Almost everywhere convergence of the gradients.}
We give here a short proof of the statement about the almost everywhere convergence of the gradients in the previous Paragraph. We only give a hint of the classic proof and refer to \cite{BocGal2, BocDGO, Kuusi10} for more details.

\vs

Take two concentric cylinders $\tilde{\mathcal Q}\Subset\mathcal Q\Subset\Omega_T\cap \{|w-a|\geq 2\sigma\}$, two functions $w_j,w_k$ of the sequence $\{w_i\}$ and a ``small'' number $\varsigma>0$. We test respectively the weak formulations of \eqref{plap} for $w_j$ and $w_k$ with the functions
\begin{equation}\label{testtruncation}
\varphi_\mp(x,t):=(\varsigma\mp T_\varsigma(w_j-w_k))\phi^p, \quad\ \text{where}\quad\  T_{\varsigma}(s):=\min\big\{\max\{s,-\varsigma\},\varsigma\big\}
\end{equation}
is the usual truncation function of \cite{BocGal2}, $\phi \in V^{2,p}_{\loc}(\mathcal Q)$ with $\partial_t\phi\in L^2_{\loc}(\mathcal Q)$, $\phi(\cdot,T)\equiv0$ and $\phi\equiv 1$ in $\bar{\mathcal Q}$. Note that $\varphi_\mp$ are admissible since $s\to T_{\varsigma}(s)$ is a Lipschitz mapping and that this is actually a formal choice, due to the fact that these test functions do not have the needed time regularity. However, in \cite[Proof of Theorem 5.3]{Kuusi10} it is shown how to appropriately perform this delicate double limiting procedure. Note that using the bi-Lipschitz relation \eqref{bilipischitz} we infer, from \eqref{p.laplacian.assumptions} and \eqref{continuityA}, that
\begin{equation}\label{split.widetilde}
\begin{split}
\langle\bar{\mathcal{A}}(x,t,u,\xi)-\bar{\mathcal{A}}(x,t,u,\zeta),\xi-\zeta\rangle>0\,,\qquad\qquad\quad  \\[3pt]
\sup_{(x,t)\in\Omega_T}|\bar{\mathcal A}(x,t,u,\xi)- \bar{\mathcal A}(x,t,v,\xi)|\leq \bar K(M,\tilde M)\,\omega_{\bar{\mathcal A},u}\big(|u-v|\big)
\end{split}
\end{equation}
for almost every $(x,t) \in \Omega_T$ and for all $(u,v,\xi,\zeta) \in\R^{2(n+1)}$, with $\zeta\neq\xi$ for the monotonicity condition and $|u|+|v|\leq M$ and $|\xi|\leq \tilde M$ for the continuity one; note that we also used the concavity of $\omega_{\mathcal A,\xi}$. Indeed
\begin{align*}
\omega_{\mathcal A,\xi}\biggl(|\xi|\Bigl|\frac{1}{\beta'(\beta^{-1}(u))}-\frac{1}{\beta'(\beta^{-1}(v)))}\Bigr|\biggr)&\leq \tilde M\omega_{\mathcal A,\xi}\big(\Lambda^2\bar K(M)\omega_{\beta'}(\Lambda|u-v|)\big)\\
&\leq \tilde \Lambda^3M\bar K(M)\omega_{\mathcal A,\xi}\big(\omega_{\beta'}(|u-v|)\big)\,,
\end{align*}
where $\bar K(M)\omega_{\beta'}$ is clearly the concave modulus of continuity for $\beta'$ when $u,v$ vary in the compact set $|u|+|v|\leq M$. Now this choice, after some algebraic manipulations (performed in detail in the aforementioned Proof), leads to
\begin{align}\label{small.varsigma}
& \int_{\mathcal Q\cap \{|w_j-w_k|\leq\varsigma\}}\langle\bar{\mathcal A}(\cdot,\cdot,w_j,Dw_j)-\bar{\mathcal A}(\cdot,\cdot,w_k,Dw_k),D(w_j-w_k)\rangle\phi^p\dx\dt\notag\\
&\quad\leq- \int_{\mathcal Q\cap \{|w_j-w_k|\leq\varsigma\}}\langle\bar{\mathcal A}(\cdot,\cdot,w_j,Dw_j)-\bar{\mathcal A}(\cdot,\cdot,w_k,Dw_k),D\phi^p\rangle(w_j-w_k)\dx\dt\notag\\
&\quad\quad\quad+c\,\varsigma+\varsigma \int_{\mathcal Q\cap \{|w_j-w_k|\leq\varsigma\}}\langle\bar{\mathcal A}(\cdot,\cdot,w_j,Dw_j)+\bar{\mathcal A}(\cdot,\cdot,w_k,Dw_k),D\phi^p\rangle\dx\dt\notag\\
&\quad\leq c\,\varsigma,
\end{align}
where the constant ultimately depends upon $p,\Lambda,|\mathcal Q|,\sup_{\partial_p\Omega_T}|g|,\phi$ (hence on $\bar{\mathcal Q}$ and $\mathcal Q$) but not on $j,k$. In the last inequality we took into account the growth condition in \eqref{tilde inf}, the standard energy estimate for the $p$-Laplacian equation \eqref{plap} together with the uniform bound \eqref{unibound}.

\vs

The goal here is to prove that the sequence $\{Dw_i\}_{i\in\N}$ converges in measure, being a Cauchy sequence with respect to this convergence. This together with the fact that the gradients are uniformly bounded in the $L^p$ norm -- and this follows again by the Caccioppoli's estimate and \eqref{unibound} -- would then lead to the needed almost everywhere convergence. To this aim we define, for $\varsigma$ as above and $\rho,\lambda>0$ the sets
\begin{align*}
E_{j,k}^\rho&:= \big\{z\in\tilde{\mathcal Q}:|Dw_j(z)-Dw_k(z)|\geq\rho\big\};\\
U_{j,k}^\varsigma&:=\big \{z\in\tilde{\mathcal Q}:|w_j(z)-w_k(z)|\leq\varsigma\big\};\\
V_{j,k}^\lambda&:=\big \{z\in\tilde{\mathcal Q}:\max\{|Dw_j(z)|,|Dw_k(z)|\}\leq\lambda\big\}.
\end{align*}
First note that, enlarging appropriately the domains of integration, we infer
\begin{align}\label{one.passage.more}
 \int_{E_{j,k}^\rho\cap U_{j,k}^\varsigma\cap V_{j,k}^\lambda}&\langle\bar{\mathcal A}(\cdot,\cdot,w_j,Dw_j)-\bar{\mathcal A}(\cdot,\cdot,w_j,Dw_k),D(w_j-w_k)\rangle\phi^p\dx\dt\notag\\ 
& \hspace{-1cm}\leq \int_{\mathcal Q\cap\{|w_j-w_k|\leq\varsigma\}}\langle\bar{\mathcal A}(\cdot,\cdot,w_j,Dw_j)-\bar{\mathcal A}(\cdot,\cdot,w_k,Dw_k),D(w_j-w_k)\rangle\phi^p\dx\dt\notag\\
&+|\mathcal Q|\bar K(2\|g\|_{L^\infty},\lambda)\omega_{\bar{\mathcal A},u}(\varsigma)\notag\\
&\hspace{-1cm}\leq c\,\big(\varsigma+\bar K(2\|g\|_{L^\infty},\lambda)\omega_{\bar{\mathcal A},u}(\varsigma)\big)
\end{align}
by \eqref{split.widetilde} together with \eqref{unibound} and \eqref{small.varsigma}, for an appropriate test function equal to one on $\tilde{\mathcal Q}$. In order to prove that the sequence $\{Dw_i\}_{i\in\N}$ is a Cauchy sequence with respect to the convergence in measure, that is, that for any $\rho>0$, once we fix $\epsilon>0$ we can find $\bar n\equiv \bar n(\epsilon)$ such that $|E_{j,k}^\rho|\leq \epsilon$ for all $j,k\geq\bar n$, we then split
\[
|E_{j,k}^\rho|\leq |E_{j,k}^\rho\cap U_{j,k}^\varsigma|+|\tilde{\mathcal Q}\smallsetminus U_{j,k}^\varsigma|\leq |E_{j,k}^\rho\cap U_{j,k}^\varsigma\cap V_{j,k}^\lambda|+|\tilde{\mathcal Q}\smallsetminus U_{j,k}^\varsigma|+|\tilde{\mathcal Q}\smallsetminus V_{j,k}^\lambda|
\]
for appropriate $\varsigma,\lambda>0$ that will be chosen in the follwing lines. Notice now that since $\{w_i\}$ converges uniformly, then in particular it is a Cauchy sequence in $\tilde{\mathcal Q}$ and hence $|\tilde{\mathcal Q}\smallsetminus U_{j,k}^\varsigma|\leq \epsilon/3$ provided we take $j,k$ large enough. Moreover, since the sequence $\{Dw_i\}$ is bounded in $L^p$, then $|\tilde{\mathcal Q}\smallsetminus V_{j,k}^\lambda|\leq \epsilon/3$ for $\lambda$ large enough; hence we can restrict now our attention on the set $E_{j,k}^\rho\cap U_{j,k}^\varsigma\cap V_{j,k}^\lambda$. We consider the set
\[
\mathcal K_{(x,t)}^{\lambda, \rho}:=\big\{(u,\xi,\zeta)\in\R^{2n+1}:|u|\leq \|g\|_{L^\infty},|\xi|,|\zeta|\leq \lambda,\ |\xi-\zeta|\geq\rho\big\} 
\]
(notice that $\lambda$ and $\rho$ are fixed) and we consider, for $(x,t)\in\tilde{\mathcal Q}$, the function 
\[
\gamma(x,t):=\inf_{\mathcal K_{(x,t)}^\lambda}\langle\bar{\mathcal A}(x,t,u,\xi)-\bar{\mathcal A}(x,t,u,\zeta),\xi-\zeta \rangle;
\]
by the continuity of $(u,\xi)\mapsto\bar{\mathcal A}(\cdot,\cdot,u,\xi)$, the compactness of $\mathcal K_{(x,t)}^{\lambda, \rho}$ and the monotonicity of $\bar{\mathcal A}$ in \eqref{split.widetilde}, we infer that $\gamma(x,t)>0$ for almost every $(x,t)\in \tilde{\mathcal Q}$. By \eqref{one.passage.more} we then have
\begin{align*}
|E_{j,k}^\rho\cap U_{j,k}^\varsigma\cap V_{j,k}^\lambda|\mean{E_{j,k}^\rho\cap U_{j,k}^\varsigma\cap V_{j,k}^\lambda}\gamma\dx\dt&\leq\int_{E_{j,k}^\rho\cap U_{j,k}^\varsigma\cap V_{j,k}^\lambda}\gamma\dx\dt\\
&\leq c\,\big(\varsigma+\bar K(2\|g\|_{L^\infty},\lambda)\omega_{\bar{\mathcal A},u}(\varsigma)\big)
 \end{align*}
and since $\gamma>0$ a.e. in $\tilde{\mathcal Q}$, we conclude with $|E_{j,k}^\rho\cap U_{j,k}^\varsigma\cap V_{j,k}^\lambda|\leq \epsilon/3$ for $\varsigma$ small enough (recall $\lambda$ has been already fixed). Hence we have proved that $\{Dw_i\}$ is a Cauchy sequence with respect to the convergence in measure. This, together with the convergence of $\{w_i\}$, yields that $\{Dw_i\}$ actually converges to $Dw$ in measure and hence there exists a subsequence which converge almost everywhere. Since the argument above actually holds for every subsequence, then almost everywhere convergence takes place  for the full sequence $\{Dw_i\}$. Finally, since $\tilde{\mathcal Q}$ is an arbitrary compactly contained subset of $\Omega_T\cap \{|w-a|\geq 2\sigma\}$ and the whole sequence converges almost everywhere, we have almost everywhere convergence in the whole $\Omega_T\cap \{|w-a|\geq 2\sigma\}$. The fact that $Dw \in L^p(\Omega_T\cap \{|w-a|\geq 2\sigma\})$ now simply follows by Lebesgue's dominated convergence Theorem and by the fact that the sequence $Dw_i$ is equibounded in $L^p$.

\subsection{Convergence near the jump}
In order to infer information on the behavior of the gradient of the approximating solutions in the set close to $\{w=a\}$, we (formally) test the equation~\eqref{final.equation} with the function 
$\varphi=T_{2\sigma}(w_i-a)\phi^p$, for some fixed $\sigma\in(0,1)$, where $T_{2\sigma}$ has been defined in \eqref{testtruncation} and $\phi \in C^\infty_c(\Omega_T)$. The rigorous treatment needs a mollification in time, see~\cite{BKU} for details. We note that since
\[
D\varphi=Dw_i\chi_{\{|w_i-a|\leq 2\sigma\}}\phi^p+p\,\phi^{p-1}T_{2\sigma}(w_i-a)D\phi,
\]
we have, with $\mathcal H$ as in \eqref{Acca}
\begin{multline*}
 \int_{\Omega_T\cap \{|w_i-a|\leq 2\sigma\}}{|Dw_i|}^p\phi^p\dx\dt \leq c(p,\Lambda)\sigma^p\,\|D\phi\|_{L^p}^p\\
+c(p,\Lambda)\int_{\Omega_T}   \int_a^{w_i} \mathcal H^\prime (\xi) T_{2\sigma}(\xi-a)\, d\xi\, \partial_t(\phi^p) \dx \dt.
\end{multline*}
For the parabolic term we formally have
\begin{align*}
\int_{\Omega_T} \partial_t w_i \,\mathcal H' (w_i) &T_{2\sigma}(w_i-a) \phi^p \dx \dt \\
&=\int_{\Omega_T} \partial_t \biggl[ \int_a^{w_i} \mathcal H^\prime  (\xi) T_{2\sigma}(\xi-a)\, d\xi \biggr]
\phi^p \dx \dt \\
&=\int_{\Omega}   \Big[\int_a^{w_i} \mathcal H^\prime (\xi) T_{2\sigma}(\xi-a)\, d\xi\,\phi^p \Big](\cdot,T)\dx\\
&\qquad\qquad-\int_{\Omega_T}   \int_a^{w_i} \mathcal H^\prime (\xi) T_{2\sigma}(\xi-a)\, d\xi\, \partial_t(\phi^p) \dx \dt.
\end{align*}
We can discard the first term since
\[
\int_a^{w_i(x,T)} \mathcal H^\prime (\xi) T_{2\sigma}(\xi-a)\, d\xi \geq0,
\]
the mapping $\xi\mapsto T_{2\sigma}(\xi-a)$ being odd with respect to the point $\xi=a$ and $\mathcal H^\prime(\cdot)\geq0$. Moreover,
\begin{align*}
\int_{\Omega_T}   \int^{w_i}_a \mathcal H^\prime (\xi) &T_{2\sigma}(\xi-a)\, d\xi\, \partial_t(\phi^p) \dx \dt\\
&\leq c(n,p)\sigma \|\partial_t(\phi^p)\|_{L^1(\Omega_T)}\,\big(\|w_i\|_{L^\infty(Q_{2R})}+a_++1\big)\\
&\leq c\big(n,p)\sigma \|\partial_t(\phi^p)\|_{L^1(\Omega_T)}\,\big(\|\beta(g)\|_{L^\infty(\partial_p\Omega_T)}+a_++1\big) 
\end{align*}
by the fact that $|T_{2\sigma}(\cdot)|\leq2\sigma$ and \eqref{unibound}. Thus, also in view of \eqref{unibound}, we can finally estimate
\begin{equation}\label{big.effort}
\int_{\Omega_T\cap \{|w_i-a|\leq 2\sigma\}}{|Dw_i|}^p\phi^p\dx\dt\leq c\,\sigma,
\end{equation}
where $c$ depends on $n,p,\Lambda,a,\|g\|_{L^\infty(\partial_p\Omega_T)},\beta$, and the test function $\phi$. 

\subsection{Passing to the limit}
We now want to pass to the limit in the weak formulation of \eqref{final.equation}; this, once fixed $\sigma\in(0,1)$, $\mathcal K\Subset\Omega$, $[t_1,t_2]\subset [0,T]$ and a test function $\varphi$ as in Definition \ref{deff}, reads as
\begin{multline}\label{app.ML}
0= \int_{\mathcal K_{t_1,t_2}\cap\{|w_i-a|>\sigma\}}\langle\bar{\mathcal A}(\cdot,\cdot,w_i,Dw_i),D\varphi\rangle\dx\dt\\
- \int_{\mathcal K_{t_1,t_2}}\big[w_i+ H_{a,\eps_i}(w_i)\big]\partial_t\varphi\dx\dt\\
+ \int_{\mathcal K_{t_1,t_2}\cap\{|w_i-a|\leq\sigma\}}\langle\bar{\mathcal A}(\cdot,\cdot,w_i,Dw_i),D\varphi\rangle\dx\dt\\
 +\int_{\mathcal K}\big\{\big[w_i+ H_{a,\eps_i}(w_i)\big]\,\varphi\big\}(\cdot,\tau) \dx\biggr|_{\tau=t_1}^{t_2}.
\end{multline}
By the continuity of $(\mu,\xi)\mapsto\mathcal A(\cdot,\cdot,\mu,\xi)$, the first term converges to 
\[
 \int_{\mathcal K_{t_1,t_2}\cap\{|w-a|>\sigma\}}\langle \mathcal A(\cdot,\cdot,u,Du),D\varphi\rangle\dx\dt
\]
as $i\to \infty$; indeed $\bar{\mathcal A}(\cdot,\cdot,w_i,Dw_i)=\mathcal A(\cdot,\cdot,u_{i},Du_{i})$. The second and the fourth terms converge to 
\[
 - \int_{\mathcal K_{t_1,t_2}}\xi\,\partial_t\varphi\dx\dt +\int_{\mathcal K} \left[ \xi\,\varphi \right] (\cdot,\tau) \dx\biggr|_{\tau=t_1}^{t_2},
\]
where $\xi$ belongs to the graph $\beta(u)+H_a(\beta(u))$ (in particular, $\xi=1/2$ if $\beta(u)=0$). Finally, by uniform convergence we can find $\bar n$, depending on $\omega$ and $\sigma$, such that 
\[
{\mathcal K_{t_1,t_2}\cap\{|w-a|\leq\sigma\}}\subset {\mathcal K_{t_1,t_2}\cap\{|w_i-a|\leq2\sigma\}}\qquad \text{for any $i\geq\bar n$.}
\]
Hence we can bound, for $i\geq\bar n$,
\begin{align}\label{Tdelta}
|\mathcal T_{i,\sigma}|:=\biggl|&\int_{\mathcal K_{t_1,t_2}\cap\{|w-a|\leq\sigma\}}\langle\bar{\mathcal A}(\cdot,\cdot,w_i,Dw_i),D\varphi\rangle\dx\dt\biggr| \notag\\
&\qquad\leq c(\Lambda) {\|D\varphi\|}_{L^p(\mathcal K_{t_1,t_2})}\bigg(\int_{\Omega_T\cap\{|w_i-a|\leq2\sigma\}}{|Dw_i|}^p\chi_{{\rm supp}\,\varphi}\dx\dt\bigg)^{1/p'}\notag\\
&\qquad\leq c\, {\|D\varphi\|}_{L^p(\mathcal K_{t_1,t_2})}\biggl(\int_{\Omega_T\cap\{|w_i-a|\leq2\sigma\}}{|Dw_i|}^p\phi^p\dx\dt\biggr)^{1/p'}\notag\\
&\qquad\leq c\,{\|D\varphi\|}_{L^p(\mathcal K_{t_1,t_2})}\sigma^{1/p'}
\end{align}
by \eqref{big.effort}, with $c\equiv c(p,\Lambda)$; here we must take $\phi$ as an appropriate cut-off function, equal to one on the support of $\varphi$. Hence if we pass to the limit (superior) $i\to \infty$ in \eqref{app.ML} with $\sigma\in(0,1)$ fixed, by rearranging terms we get
\begin{multline*}
- \int_{\mathcal K_{t_1,t_2}}\xi\,\partial_t\varphi\dx\dt+ \int_{\mathcal K_{t_1,t_2}\cap\{|\beta(u)- a|>\sigma\}}\langle {\mathcal A}(\cdot,\cdot,u,Du),D\varphi\rangle\dx\dt+\mathcal T_\sigma\\
 +\int_{\mathcal K} \left[ \xi\,\varphi \right] (\cdot,\tau) \dx\biggr|_{\tau=t_1}^{t_2}=0, 
\end{multline*}
with $\xi \in \beta(u)+H_a(\beta(u))$, since the equibounded sequence $\mathcal T_{i,\sigma}$ converges to a limit (possibly only superior) $\mathcal T_\sigma$ such that $|\mathcal T_\sigma|\leq c\,\sigma^{1/p'}$ by \eqref{Tdelta}. Now we take the limit $\sigma\to0$ and we get
\begin{multline*}
-\int_{\mathcal K_{t_1,t_2}}\xi\,\partial_t\varphi\dx\dt+\int_{\mathcal K_{t_1,t_2}\cap\{\beta(u)\neq a\}}\langle{\mathcal A}(\cdot,\cdot,u,Du),D\varphi\rangle\dx\dt\\
 +\int_{\mathcal K} \left[ \xi\,\varphi \right] (\cdot,\tau) \dx\biggr|_{\tau=t_1}^{t_2}=0.
\end{multline*}
Using well-known properties of Sobolev functions and \eqref{bilipischitz}, the second integral is equal to the integral over $\mathcal K_{t_1,t_2}$ of the same function. Hence, we have proved that the pointwise limit defined in Paragraph \ref{AAArgument} is a local weak solution to \eqref{the general equation} in the sense of Definition \ref{deff}.

\vspace{5mm}

{\small {\bf Acknowledgments:} This paper was partially conceived while the authors were part of the research program ``Evolutionary Problems'' at the Institut Mittag-Leffler (Djursholm, Sweden) in the Fall 2013. The support, the hospitality and the optimal environment of the Institut is gratefully acknowledged. The paper was concluded after visits of PB and JMU to Aalto University and of TK to the University of Coimbra. The authors are grateful to both institutions.

PB has been supported by the Gruppo Nazionale per l'Analisi Matematica, la Probabilit\`a e le loro Applicazioni (GNAMPA) of the Istituto Nazionale di Alta Matematica (INdAM). TK and CL were supported by the Academy of Finland project ``Regularity theory for nonlinear parabolic partial differential equations". CL has also been supported by the Vilho, Yrj\"o and Kalle V\"ais\"al\"a Foundation. JMU was partially supported by CMUC -- UID/MAT/00324/2013, funded by the Portuguese Government through FCT/MCTES and co-funded by the European Regional Development Fund through the Partnership Agreement PT2020.}


\begin{thebibliography}{99}

\bibitem{BKU}
{\sc P. Baroni, T. Kuusi, J.M. Urbano}:
A quantitative modulus of continuity fot the two-phase Stefan problem,
{\em Arch. Rational Mech. Anal.} {\bf 214} (2014), no. 2, 545--573.

\bibitem{BocGal2} {\sc L.~Boccardo and T. Gallou\"et}:
Nonlinear elliptic and parabolic equations involving measure data,
{\em J. Funct. Anal.} \textbf{87} (1989), no. 1, 149--169.

\bibitem{BocDGO} {\sc L. Boccardo, A. Dall'Aglio, T. Gallou\"et, L. Orsina}:
Nonlinear parabolic equations with measure data,
{\em J. Funct. Anal.} \textbf{147} (1997), no.1, 237--258.

\bibitem{DiB2}
{\sc E. DiBenedetto}:
Continuity of weak solutions to certain singular parabolic equations,
{\em Ann. Mat. Pura Appl. {\rm (4)}} {\bf 103} (1982), 131--176.

\bibitem{DiB3}
{\sc E. DiBenedetto}:
A boundary modulus of continuity for a class of singular parabolic equations,
{\em J. Differential Equations} {\bf 63} (1986), no. 3, 418--447.

\bibitem{DiBe93}
{\sc E. DiBenedetto}:
{\em Degenerate parabolic equations}, 
Universitext, Springer-Verlag, New York, 1993.

\bibitem{DiBePDE}
{\sc E. DiBenedetto}:
{\em Partial differential equations}, 
Birkh\"auser Boston, Boston, MA, 1995.

\bibitem{DiBeUrbaVesp04}
{\sc E. DiBenedetto, J.M. Urbano and V. Vespri}:
{\em Current issues on singular and degenerate evolution equations},
Evolutionary equations. Vol. I, Handb. Differ. Equ., 169--286, North-Holland, Amsterdam, 2004.

\bibitem{EvansGariepy} {\sc L. C. Evans, R. F. Gariepy}: 
	\newblock {\em Measure Theory and Fine Properties of Functions}, 
	\newblock Stud. Adv. Math., CRC Press, Boca Raton, FL, 1992.

\bibitem{Kuusi10}
{\sc R. Korte, T. Kuusi and M. Parviainen}:
A connection between a general class of superparabolic functions and supersolutions, 
{\em J. Evol. Equ.} {\bf 10} (2010), no. 1, 1--20.

\bibitem{KMWolff} 
{\sc T.~Kuusi, G.~Mingione}:
The Wolff gradient bound for degenerate parabolic equations,
{\em J. Eur. Math. Soc. {\rm(}JEMS\,{\rm)}} {\bf 16} (2014), no. 4, 835--892

\bibitem{Noche87}
{\sc R.H. Nochetto}:
A class of nondegenerate two-phase Stefan problems in several space variables,
{\em Comm. Partial Differential Equations} {\bf 12} (1987), no. 1, 21--45.

\bibitem{Salsa12}
{\sc S. Salsa}:
Two-phase Stefan problem. Recent results and open questions.
{\em Milan J. Math.} {\bf 80} (2012), no. 2, 267--281.

\bibitem{UCon}
{\sc J.M. Urbano}:
Continuous solutions for a degenerate free boundary problem,
{\em Ann. Mat. Pura Appl. {\rm (4)}} {\bf 178} (2000), 195--224. 

\bibitem{Urba08}
{\sc J.M. Urbano}:
{\em The method of intrinsic scaling, A systematic approach to regularity for degenerate and singular PDEs},
Lecture Notes in Mathematics, 1930, Springer-Verlag, Berlin, 2008.

\bibitem{Ziemer}
{\sc W.P. Ziemer}:
Interior and boundary continuity of weak solutions of degenerate parabolic equations,
{\em Trans. Amer. Math. Soc.} {\bf 271} (1982), no. 2, 733--748. 
\end{thebibliography}
\end{document}